\documentclass{siamltex}
\usepackage{amsfonts, amsmath, amssymb, psfrag, graphicx, bbm, mathrsfs}
\usepackage{mathabx}   % fuer \widecheck

\usepackage[normalem]{ulem}   % for sout{..}     % \usepackage{refcheck}

\usepackage[colorlinks=true]{hyperref}
\hypersetup{urlcolor=blue, citecolor=blue, linkcolor=blue}

\newtheorem{remark}[theorem]{Remark}

\def\fin{\ifmmode{\Large$\diamond$}\else{\unskip\nobreak\hfil
    \penalty50\hskip1em\null\nobreak\hfil{\Large$\diamond$}
    \parfillskip=0pt\finalhyphendemerits=0\endgraf}\fi}

\def\be#1#2\ee{\begin{equation}\label{eq:#1}#2\end{equation}}
\def\req#1{{\rm(\ref{eq:#1})}}
\def\bdm  {\begin{displaymath}}
  \def\edm  {\end{displaymath}}
\def\bdmal{\begin{displaymath}\begin{aligned}}
    \def\edmal{\end{aligned}\end{displaymath}}

\mathchardef\PhiG="0108
\mathchardef\PsiG="0109
\mathchardef\Omega="010A
\mathchardef\Sigma="0106
\mathchardef\Gamma="0100
\mathchardef\Lambda="0103

\newcommand{\N}{{\mathord{\mathbb N}}}
\newcommand{\R}{{\mathord{\mathbb R}}}

\renewcommand{\S}{W^{1,\infty}(\partial\D)}  
\newcommand{\Sp}{W^{1,\infty}(\partial\D')}  
\newcommand{\W}{{W_c^{1,\infty}(\Omega)}}

\newcommand{\B}{{\cal B}}

\newcommand{\CT}{{C_\D}}

\newcommand{\CTp}{{C_{\D'}}}
\newcommand{\norm}[1]{\|#1\|}

\newcommand{\scalp}[1]{\langle\,#1\,\rangle}

\newcommand{\rmd}{\,\mathrm{d}}
\newcommand{\ds}{\rmd s}

\newcommand{\dx}{\rmd x}

\newcommand{\ytilde}{{\widetilde y}}

\def\req#1{{\rm(\ref{eq:#1})}}

\makeatletter
\newcommand{\dupdots}{\mathinner{\mkern1mu\raise\p@
    \vbox{\kern7\p@\hbox{.}}\mkern2mu
    \raise4\p@\hbox{.}\mkern2mu\raise7\p@\hbox{.}\mkern1mu}}
\makeatother
\makeatletter
\newcount\c@MaxMatrixCols \c@MaxMatrixCols=10
\newenvironment{cmatrixc}{%
  \hskip -\arraycolsep\array{*\c@MaxMatrixCols c}%
}{%
  \endarray \hskip -\arraycolsep
}
\makeatother
\newenvironment{cmatrix}{\left[\cmatrixc}{\endmatrix\right]}
\def\thsp{\hspace*{0.1ex}}

 {\endMakeFramed}

\newcommand{\upunkt}{{\dot{u}}}
\newcommand{\utilde}{{\widetilde u}}

\newcommand{\D}{{\mathscr D}}
\newcommand{\Q}{{\mathscr Q}}

\newcommand{\dtheta}{\rmd \theta}

\makeatletter
\renewcommand\@biblabel[1]{#1.}
\makeatother

\title{On the shape derivative of polygonal inclusions in the conductivity 
problem}
\author{Martin Hanke\thanks{Institut f\"ur Mathematik, Johannes
    Gutenberg-Universit\"at Mainz, 55099 Mainz, Germany
    ({\tt hanke@math.uni-mainz.de}).}}

\begin{document}
\sloppy
\maketitle

\begin{abstract}
We consider the conductivity problem for a homogeneous body with an
inclusion of a different, but known, conductivity. 
Our interest concerns the associated shape derivative, i.e., 
the derivative of the corresponding electrostatic potential with respect to 
the shape of the inclusion. For a smooth inclusion it is
known that the shape derivative is the solution of a specific inhomogeneous
transmission problem. We show that this characterization of the
shape derivative is also valid when the inclusion is a polygonal domain,
but due to singularities at the vertices of the polygon, the
shape derivative fails to belong to $H^1$ in this case.
\end{abstract}


\begin{AMS}
  {\sc 35B65, 49N60, 49Q10, 65N21}
\end{AMS}

%\hspace*{-0.7em}
%{\footnotesize \textbf{Last modified.} \today}

%\pagestyle{myheadings}
%\thispagestyle{plain}
%\markboth{M. HANKE}
%{ITERATIVE BOLTZMANN INVERSION}
%
%\addtocounter{footnote}{1}

%%%%%%%%%%%%%%%%%%%%%%%%%%%%%%%%%%%%%%%%%%%%%%%%%%%%%%%%%%%%%%%%%%%%%%%
\section{Introduction}
\label{Sec:Intro}
%%%%%%%%%%%%%%%%%%%%%%%%%%%%%%%%%%%%%%%%%%%%%%%%%%%%%%%%%%%%%%%%%%%%%%% 
The inverse conductivity problem aims at information 
about the spatial electric conductivity distribution within a body, 
which is only accessible to electrostatic or quasi-electrostatic 
measurements at its boundary. 
For example, in electric impedance tomography a series of probing currents 
is applied through electrodes attached to the surface of the body, 
and the resulting boundary potentials are measured (or the other way round).
As shown by Astala and P\"aiv\"arinta~\cite{AsPa06}
the full set of these current/voltage pairs completely determines an 
isotropic conductivity distribution in two space dimensions. 
See the monographs by Mueller and Siltanen~\cite{MuSi12} and 
Kirsch~\cite{Kirs11} for an exposition of electric impedance tomography. 
Except for the D-bar~\cite{KLMS07} 
and the layer stripping methods~\cite{SCII91} most of the pertinent 
inversion algorithms are iterative in nature:
The spatial conductivity in question is represented by an element 
from the nonnegative cone of a reasonable infinite dimensional vector space 
of functions, and in each iteration the gradient of some loss function 
is used to update the current approximation of the conductivity distribution;
cf., e.g., \cite{CINGS90,Dobs92,PLP95} for some early references. 
This is the method of choice when the material is very inhomogeneous.

When the object is homogeneous, except for an inclusion of some other known
material say, then one can hope that only one or a few pairs of Cauchy data
may suffice to determine the location and the shape of the inclusion. 
Sophisticated noniterative reconstruction methods
like pole fitting methods~\cite{BLMS99,KaLe04,Hank08} 
%the enclosure method~\cite{Ikeh00}, the factorization method~\cite{Brue01},
or the computation of the convex source support~\cite{HHR08} have been 
developed for this setting, but iterative methods can also be an option, 
provided they 
%but in this case they should 
exploit gradient information with respect to the shape of
the inclusion (or its parameterization).
%rather than a spatially distributed
%conductivity function, and utilize an appropriate concept for the derivative
%of the boundary data with respect to the domain of the inclusion. 

Such derivatives have been termed \emph{shape derivative} or 
\emph{domain derivative}, and have originally been developed 
in the optimization community, cf., e.g., \cite{SoZo92,HePi05}. 
The idea is to analyze the impact
of a pointwise perturbation of the boundary of the inclusion in the direction 
of a vector field that is attached to it. Then a gradient descent or Newton
type method can be used to shift and deform an approximate inclusion so as
to match the given data as good as possible.

Shape derivatives for inverse conductivity problems have been determined 
in the late 90's of the previous century.
However, up to recently their rigorous foundation
has been limited to inclusions with a certain smoothness,
e.g., $C^{1,1}$-domains. For this smooth case it has been shown 
by Hettlich and Rundell~\cite{HeRu98} that the shape derivative of an
inclusion with prescribed conductivity can be characterized via the solution 
of a transmission problem with distributed sources sitting on the boundary of
the inclusion: The solution of this differential equation provides
the perturbation of the electric potential at each individual point
within and on the boundary of the body under consideration.
More details and references follow in the subsequent section.

In the optimization community shape derivatives for less regular domains, 
e.g., Lipschitz domains,
have been studied, e.g., by Delfour and Zol\'{e}sio~\cite{DeZo92}, 
Lamboley, Novruzi, and Pierre~\cite{LNP16}, and Laurain~\cite{Laur20};
these works focus on the derivative of certain integrated shape functionals.
%For less regular domains little was
%known until recently, when 
In the application that we are interested in, such a functional can be, 
for example, the inner product of the given data with a certain test function
on the boundary of the two-dimensional body. For this particular
case Beretta, Francini, and Vessella~\cite{BFV17} have determined 
%studied polygonal inclusions in the two-dimensional case.
in a \emph{tour de force} 
%they determined 
the impact of a perturbation of a conducting polygonal inclusion.
% on electrostatic measurements on the surface of the body.
%the shape derivative of a polygonal inclusion. 
In a subsequent paper, 
Beretta, Micheletti, Perotto, and Santacesaria~\cite{BMPS18} 
gave a somewhat more elegant derivation of the same result 
using the established shape derivative theory.
%theory from the smooth case. 
Still, this functional only provides the shape derivative of the given data
in a weak (i.e., variational) form, which lacks
the interpretation of \cite{HeRu98} via an underlying transmission problem. 
Further, the impact of the perturbation of the polygon on the electric 
potential in the interior of the body (i.e., the state variable)
remained unresolved.

In this paper we show that the same transmission problem as in the smooth
case describes the sensitivity of the electric potential at each point 
within the body; its trace on the outer boundary provides an explicit 
definition of the shape derivative of the electrostatic measurements. 
Due to singularities at the vertices of the polygon, the solution of this
transmission problem is less regular than in the smooth case. 
This presents some difficulties in proving the existence of a solution, 
which can be overcome by using a technique
which goes back to Kondratiev~\cite{Kond67} (see also Grisvard~\cite{Gris92}),
and which has also been employed in \cite{Laur20}.
This technique is tailored to the two-dimensional case, but it may be
possible to follow the approach from \cite{Gris92} to also treat the case of
polyhedral inclusions in three dimensional domains.

Our results are relevant for investigating the stability of
numerical algorithms for solving the inverse problem. 
As we show in \cite{Hank24} the outcome of the present work implies
that the determination of a polygon from 
two pairs of current/voltage boundary measurements 
-- which is possible according to Seo~\cite{Seo96} -- is actually well-posed, 
i.e., the inverse operator is locally Lipschitz continuous.
To the best of our knowledge, a proof of this result on the grounds 
of the aforementioned weak representation of the shape derivative by 
Beretta et al.\ does not seem to be possible.
Of course, our explicit representation of the shape derivative can also
be exploited to extend the Newton-type scheme in \cite{HeRu98}, for example,
to polygonal inclusions.

The outline of this paper is as follows. In Section~\ref{Sec:background}
we set up the problem under consideration and review known results about 
the associated shape derivative. 
After that we focus on polygonal inclusions only. 
In Section~\ref{Sec:forward} we recapitulate the properties of the 
electric potential for such inclusions.
Then, in Section~\ref{Sec:estimates} we improve upon some estimates 
for the shape derivative of the electrostatic boundary measurements determined 
in \cite{BFV17,BMPS18}, and provide preliminary results about the
shape derivative of the electric potential within the body. 
In Section~\ref{Sec:HeRu98} we turn to the transmission problem 
formulated in \cite{HeRu98}, and show that it has a unique solution 
when the inclusion is a polygon, and we determine its regularity.
That this solution coincides with the shape derivative of the electric
potential is the topic of Section~\ref{Sec:domderiv}. In the final 
section we briefly discuss the case of an
insulating or a perfectly conducting polygonal inclusion; all our results
essentially extend to these two degenerate cases.

%%%%%%%%%%%%%%%%%%%%%%%%%%%%%%%%%%%%%%%%%%%%%%%%%%%%%%%%%%%%%%%%%%%%%%%
\section{The conductivity problem and associated shape derivatives}
\label{Sec:background}
%%%%%%%%%%%%%%%%%%%%%%%%%%%%%%%%%%%%%%%%%%%%%%%%%%%%%%%%%%%%%%%%%%%%%%% 
Let $\Omega\subset\R^2$ be a simply connected Lipschitz domain 
and $\D$ be another Lipschitz domain with simply connected closure
$\overline\D\subset\Omega$. 
The outer normal vector on $\partial\D$ and $\partial\Omega$, respectively,
will be denoted by $\nu$.
We call $\D$ the inclusion, and we assume that the conductivity in $\Omega$ 
is given by
\be{sigma}
   \sigma(x) \,=\, \begin{cases}
                      k\,, & \text{for $x\in\D$}\,,\\
                      1\,, & \text{for $x\in\Omega\setminus\overline\D$}\,,
                   \end{cases}
\ee
where $k$ is taken to be greater than zero and different
from one; in Section~\ref{Sec:degeneratecases}, 
we will briefly treat the degenerate cases $k=0$ and $k=+\infty$, respectively.

Given a fixed (nontrivial) driving boundary current 
\bdm
   f\,\in\, L^2_\diamond(\partial\Omega)
     \,=\, \bigl\{\,f\in L^2(\partial\Omega) \,:\, 
                    \int_{\partial\Omega}f\ds=0\,\bigr\}\,,
\edm
the induced quasistatic electric potential
\bdm
   u\,\in\, H^1_\diamond(\Omega) 
   \,=\, \bigl\{\,u\in H^1(\Omega)\,:\, \int_{\partial\Omega} u\ds = 0\,\bigr\}
\edm
is the unique solution of the conductivity equation
\be{conductivity-equation}
   \nabla\cdot (\sigma \nabla u) \,=\, 0  \quad \text{in $\Omega$}\,, \qquad
   \frac{\partial}{\partial\nu} u \,=\, f \quad \text{on $\partial\Omega$}\,,
\ee
with vanishing mean on $\partial\Omega$. The restrictions $u^-=u|_\D$ and
$u^+=u|_{\Omega\setminus\overline\D}$ 
are both harmonic (smooth) functions, which satisfy the transmission conditions
\be{Bnu}
   \bigl[u\bigr]_{\partial\D} = u^+|_{\partial\D}-u^-|_{\partial\D} = 0 
   \qquad \text{and} \qquad
   \bigl[D_\nu u \bigr]_{\partial\D} = 
   \frac{\partial}{\partial\nu}u^+ \,-\, k\,\frac{\partial}{\partial\nu} u^-
   = 0
\ee
on $\partial\D$, 
where the equalities are to be understood in $H^{\pm 1/2}(\partial\D)$, 
respectively.
Alternatively, $u$ can be characterized as the unique solution in
$H^1_\diamond(\Omega)$ of the variational problem 
\be{forwardvarproblem}
   \int_\Omega \sigma\,\nabla u\cdot\nabla w\dx 
   \,=\, \int_{\partial\Omega} fw\ds \qquad \text{for all $w\in H^1(\Omega)$}\,.
\ee

Assuming that the (normalized) background conductivity and its 
anomalous value $k>0$ in the inclusion are fixed (and known), 
we are interested in the sensitivity of the boundary data 
%inverse problem of determining $\D$ from the (measured) boundary potential 
\bdm
   u|_{\partial\Omega} \,=:\, \Lambda_f(\D) \,\in\, L^2_\diamond(\partial\Omega)
\edm
%Various methods have been suggested to approach this problem, cf., e.g.,
%\cite{.,.,.}, but in this paper our emphasis is on the
of the electrostatic potential with respect to perturbations of $\D$.
To be specific, assume that a vector field $h:\partial\D\to\R^2$ with
$h\in\S$
%\bdm
%   h \,\in\, \W \,=\,
%   \bigl\{\,h\in W^{1,\infty}(\Omega)\,:\,\supp h \subset\Omega\,\bigr\}
%\edm
is prescribed, and define the perturbation
\bdm
   \Gamma_{h} \,=\, \bigl\{x+h(x)\,:\, x\in\partial\D\}
\edm
of $\partial\D$. 
%When $h$ belongs to $\S$
%can be extended to a
%vector field in a neighborhood $\U$ of $\partial\D$ with
%\bdm
%   \norm{h}_{W^{1,\infty}(\U)}
%   \,=\, \sup_{x\in\U} \norm{h'(x)}_2 \,<\, \infty\,,
%\edm
Then $\Gamma_h$ is a Jordan curve in $\Omega$ for $h$ sufficiently small, 
and we denote its interior domain by $\D_h$. 
Let $u_h$ be the solution of the corresponding 
problem~\req{conductivity-equation} with $\D$ replaced by $\D_h$; 
then the so-called shape derivative $u_h'$ of the electric potential
$u=u(\D)$ in direction $h$ is given by
\be{defdomderiv}
   u_h' \,=\, \lim_{t\to 0} \frac{u_{th}-u}{t} \qquad \text{in $\Omega$}\,,
\ee
provided this limit exists. Likewise,
\bdm
   \partial\Lambda_f(\D)h 
   \,=\, \lim_{t\to 0} \frac{1}{t}\bigl(\Lambda_f(\D_{th})-\Lambda_f(\D)\bigr)
\edm
is the shape derivative of $\Lambda_f(\D)$ in direction $h$, if the
latter limit exists.

Under the assumption that the boundary of $\D$ is smooth,
Hettlich and Rundell~\cite{HeRu98} 
(see also Afraites, Dambrine, and Kateb~\cite{ADK07}) 
established the existence of $u_h'$ and showed that it is the unique solution
of the
%potential $u'$
%\bdm
%   u' \,=\, \begin{cases}
%               u'{}^- & \text{in $\D$}\,,\\
%               u'{}^+ & \text{in $\Omega\setminus\overline\D$}\,,
%            \end{cases}
%\edm
%which satisfies the 
inhomogeneous transmission problem
\be{HeRu98}
\begin{array}{c}
   \Delta u_h' \,=\, 0 \quad \text{in $\Omega\setminus\partial\D$}\,, \qquad
   \dfrac{\partial}{\partial\nu}u_h' \,=\, 0 \quad \text{on $\partial\Omega$}\,,
   \qquad 
   {\displaystyle \int_{\partial\Omega} u_h'\ds \,=\, 0\,,}
   \\[3ex]
   \bigl[u_h'\bigr]_{\partial\D} 
   \,=\, (1-k)(h\cdot\nu)\dfrac{\partial}{\partial\nu}u^{-}\,, \quad
   \bigl[D_\nu u_h'\bigr]_{\partial\D} 
   \,=\, (1-k)\,\dfrac{\partial}{\partial\tau}
                \Bigl((h\cdot\nu)\dfrac{\partial}{\partial\tau}u\Bigr)\,.
\end{array}
\ee
%Take note that the definition of the jump brackets 
%$\bigl[\,\cdot\,\bigr]_{\partial\D}$ in \cite{HeRu98} is the opposite of ours;
%cf.~\req{Bnu}.
Due to the $H^2$-regularity of the forward problem at either side of
the smooth boundary of $\D$, compare, e.g., McLean~\cite[Theorem~4.20]{McLe00}, 
the inhomogeneous data for this problem belong
to the appropriate Sobelev spaces $H^{\pm1/2}(\partial\D)$, and this garantees
existence of a unique weak solution of \req{HeRu98}, whose
restriction to $\Omega\setminus\overline\D$ belongs to $H^1$, so that its 
trace $u'|_{\partial\Omega}\in L^2_\diamond(\partial\Omega)$ is well-defined.
If $\D$ is merely a Lipschitz domain, then this chain of arguments is no 
longer valid, and it is not immediately clear, whether the transmission 
problem~\req{HeRu98} admits a solution, and if so, in which space.
And even if there is a unique solution of \req{HeRu98}, one may still wonder
whether this solution is the shape derivative \req{defdomderiv} of $u$.

As a step towards less regular inclusions
Beretta, Francini, and Vessella~\cite{BFV17}
studied the conductivity equation~\req{conductivity-equation} with
a polygonal anomaly and a perturbation field $h$, whose vector components
are linear splines over $\partial\D$ with their nodes attached to the vertices 
of $\D$; i.e., both components of $h$ are affine linear on each of the edges 
of $\D$ and continuous at the vertices. With such a field the perturbed
domains $\D_h$ also are polygons -- at least for $h$ sufficiently small.
It is proved in \cite{BFV17} that the associated shape derivative 
$\partial\Lambda_f(\D)$ of $\Lambda_f$ at $\D$ exists in a Fr\'echet sense; 
more precisely they showed that
\be{Frechet-bound}
   \norm{\Lambda_f(\D_h) - \Lambda_f(\D) 
         - \partial\Lambda_f(\D)h}_{L^2(\partial\Omega)}
   \,\leq\, C\thsp \norm{h}^{1+\delta}
\ee
for $h$ sufficiently small, where the constants $C>0$ and $\delta\in(0,1)$ 
are independent of the particular choice of $h$. 
(Note that in \req{Frechet-bound} it is irrelevant which norm of $h$ is used, 
because the admissible spline functions from \cite{BFV17}
constitute a finite dimensional vector space.) 
Finally, the authors of \cite{BFV17} came up with a weak (variational) 
definition of this shape derivative, 
namely\footnote{Beware of the sign error in \cite[Theorem~4.6]{BFV17}.}
\be{BFV17}
   \scalp{\partial\Lambda_f(\D)h,g}_{L^2(\partial\Omega)}
   \,=\, (1-k)\int_{\partial\D} (h\cdot\nu)\,\nabla u^-\cdot(M\nabla v_g^-)\ds 
\ee
for all $g\in L^2_\diamond(\partial\Omega)$, where %$v_g\in H^1_\diamond(\Omega)$
\bdm
   v_g \,=\, \begin{cases} 
                v_g^- & \text{in $\D$}\,, \\ 
                v_g^+ & \text{in $\Omega\setminus\overline\D$}\,,
             \end{cases}
\edm
is the corresponding solution in $H^1_\diamond(\Omega)$ of the 
conductivity equation
\be{vg-pde}
   \nabla\cdot (\sigma \nabla v_g) \,=\, 0  \quad \text{in $\Omega$}\,, \qquad
   \frac{\partial}{\partial\nu} v_g \,=\, g \quad \text{on $\partial\Omega$}\,,
\ee
with driving current $g$, and $M$ is the symmetric $2\times 2$-matrix 
with eigenvalues $1$ and $k$ and corresponding eigenvectors
in tangential and normal directions, respectively;
see \req{Heps} below for a justification that the expression~\req{BFV17}
is well-defined. 

We mention that the analysis in \cite{BFV17} avoids the
use of the transmission problem~\req{HeRu98} and the general theory of 
shape derivatives. An alternative derivation of \req{BFV17}
-- valid for general vector fields $h\in\S$ --
on the grounds of this latter theory was subsequently handed in by 
Beretta, Micheletti, Perotto, and Santacesaria~\cite{BMPS18}, who determined
the so-called material derivative $\upunkt$ of $u$ in $\Omega$, 
cf.~\req{upunkt} below, but the
existence of the shape derivative $u'$ of $u$ remained unsettled.
In Section~\ref{Sec:estimates} we will also exploit the material derivative 
to show that the estimate~\req{Frechet-bound} for the Taylor remainder of the 
shape derivative of the boundary data can be improved to the order 
$\norm{h}^2$. 

In Sections~\ref{Sec:HeRu98} and \ref{Sec:domderiv} we analyze the
transmission problem~\req{HeRu98} for polygonal inclusions
as treated in \cite{BFV17,BMPS18}, and we prove that it has a 
unique solution $u_h'$. 
In contrast to the smooth case its restrictions to $\D$ and to
$\Omega\setminus\overline\D$ merely belong to some Sobolev space $H^\gamma$, 
where $\gamma\in(1/2,1)$ depends on the interior angles of the vertices 
of $\D$. We further show that $u_h'$ has a well-defined trace 
%$u'|_{\partial\Omega} \in L^2_\diamond(\partial\Omega)$, 
which satisfies the same variational problem~\req{BFV17} 
as $\partial\Lambda_f(\D)h$; from this we then conclude
that the solution of the transmission problem is the shape derivative of $u$ 
in $\Omega$, i.e.,  
that the main result by Hettlich and Rundell carries over to 
polygonal inclusions. 
%We finally establish in Section~\ref{Sec:estimates}
%that the estimate~\req{Frechet-bound} for the Taylor remainder of the 
%domain derivative can be improved to the order $t^2$, by exploiting the 
%higher regularity of the material derivative from \cite{BMPS18}.

%%%%%%%%%%%%%%%%%%%%%%%%%%%%%%%%%%%%%%%%%%%%%%%%%%%%%%%%%%%%%%%%%%%%%%%
\section{The conductivity problem with polygonal inclusions}
\label{Sec:forward}
%%%%%%%%%%%%%%%%%%%%%%%%%%%%%%%%%%%%%%%%%%%%%%%%%%%%%%%%%%%%%%%%%%%%%%% 
We recapitulate some basic facts about 
the two-dimensional conductivity problem~\req{conductivity-equation} 
with a polygonal inclusion $\D$ with $n$ vertices, $n\geq 3$,
and a simply connected closure $\overline\D\subset\Omega$. 

We start with some notation:
We denote the (relatively open) edges of $\D$ by $\Gamma_i$ and its vertices 
by $x_i$, $i=1,\dots,n$, with the convention that $x_i$ connects $\Gamma_i$ 
and $\Gamma_{i+1}$, where we identify $\Gamma_{n+1}$ with $\Gamma_1$;
likewise we identify $x_{n+1}$ with $x_1$, when necessary.
On $\Gamma_i$ the tangent vector $\tau$ is pointing in the direction of $x_i$.
The interior angle of $\D$ at $x_i$ is denoted by 
$\alpha_i\in(0,2\pi)\setminus\{\pi\}$, and we use a local coordinate system
\be{localcoordinates}
   x \,=\, x_i \,+\, 
           \bigl(r\cos(\theta_i+\theta),r\sin(\theta_i+\theta)\bigr)\,, \qquad 
   0<r<r_i\,, \ 0\leq \theta < 2\pi\,,
\ee
near $x=x_i$, where $\theta_i$ is such that $\theta=0$ corresponds to points on
$\Gamma_{i+1}$, $\theta=\alpha_i$ for $x\in\Gamma_i$, and the range 
$0<\theta<\alpha_i$ corresponds to points $x\in\D$, while the interval 
$\alpha_i<\theta<2\pi$ corresponds to points 
$x\in\Omega\setminus\overline\D$. We will also make use of 
cut-off functions $\chi_i=\chi_i(x)\in C_0^\infty(\Omega)$, 
which only depend on the distance $|x-x_i|$, are one near $x=x_i$ and
zero for $|x-x_i|\geq r_i$. Without loss of generality we assume that the
radius $r_i$ is so small that the supports of any two cut-off functions 
$\chi_i$ and $\chi_j$ with $i\neq j$ have no points in common, 
and that the intersection of $\partial\D$ with 
the support of $\chi_i$ is a subset of $\Gamma_i\cup\{x_i\}\cup\Gamma_{i+1}$.

Since the two components $u^\pm$ of the solution $u$ of 
\req{conductivity-equation} are harmonic and satisfy the
transmission conditions~\req{Bnu} they can both be continued as harmonic
functions across each of the edges $\Gamma_i$ of $\D$ by 
appropriate reflections; compare, e.g., the argument in the proof of
Lemma~5.2 in Friedman and Isakov~\cite{FrIs89}.
% across each of the edges $\Gamma_i$ of $\partial\D$.
This proves that all their derivatives extend continuously to every edge
from either side. In particular this implies that $u^+|_{\Gamma_i}=u^-|_{\Gamma_i}$
is a well-defined $C^\infty$ function; we therefore drop the superscripts $\pm$
when dealing with this trace.

We quote from Bellout, Friedman, and Isakov~\cite{BFI92} that near a fixed
vertex $x_i$ of $\D$, $i\in\{1,\dots,n\}$, the potential $u$ satisfies
\be{BFI92}
   u(x) \,=\, u(x_i) \,+\, 
              \sum_{j=1}^\infty \beta_{ij}\, y_{ij}(\theta)\,r^{\gamma_{ij}}\,,
\ee
when using the local coordinate system \req{localcoordinates} for
\bdm
   x\in\B_{r_i}(x_i) \,=\, \bigl\{x\,:\, |x-x_i|<r_i\bigr\}\,.
\edm
Together with the constant function $y_{i0}$
the functions $y_{ij}\in H^1(0,2\pi)$ in \req{BFI92} are the eigenfunctions
of the eigenvalue problem
\begin{subequations}
\label{eq:diffeq-ang}
\be{diffeq-ang-op}
   -y''=\gamma^2y \qquad \text{in $(0,2\pi)\setminus\{\alpha_i\}$}\,,
\ee
subject to the transmission conditions
\be{diff-eq-ang-transmission}
\begin{aligned}
   &\phantom{k'\,}y(0+)=y(2\pi-)\,, & \quad 
   &\phantom{k'\,}y(\alpha_i-)=y(\alpha_i+)\,,  \\[1ex]
   &k\,y'(0+)=y'(2\pi-)\,,& \quad &k\,y'(\alpha_i-)=y'(\alpha_i+)\,,
\end{aligned}
\ee
\end{subequations}   
$\gamma_{ij}^2$ are the corresponding eigenvalues in increasing order,
and $\gamma_{ij}$ its square roots.
The latter are the nonnegative solutions of the nonlinear equation
%The exponents $\gamma_{ij}$ for fixed $i=1,\dots,n$ are the infinitely 
%many positive solutions of the nonlinear equation
\be{gamma-cond}
   \bigl|\sin\gamma_{ij}(\alpha_i-\pi)\bigr| \,=\, 
   \lambda\,\bigl|\sin\gamma_{ij}\pi\bigr|\,,
   \qquad \lambda\,=\, \Bigl|\frac{k+1}{k-1}\Bigr|\,.
\ee
There holds
\be{gamma}
   \gamma_{i0}\,=\, 0\,, \quad \frac{1}{2} \,<\, \gamma_{i1} \,<\, 1\,, \quad
   \text{and} \quad 
   \gamma_{i2}\,>\, 1\,,
%   1 \,<\, \gamma_{i2} \,<\, \frac{3}{2}\,, \quad
%   \text{and} \quad 
%   \gamma_{ij} \,>\, \frac{3}{2} \quad \text{for $j \geq 3$}\,,
\ee
%moreover, $\gamma_{ij}/j$ stays away from zero and infinity, as $j\to\infty$.
and the eigenfunctions have the form
\bdm
   y_{ij}(\theta) \,=\, 
   \begin{cases}
      A_{ij}^- \cos\gamma_{ij}\theta\,+\,B_{ij}^- \sin\gamma_{ij}\theta\,,
      & \quad \phantom{\alpha_i}0<\theta<\alpha_i\,,\\[1ex]
      A_{ij}^+ \cos\gamma_{ij}\theta\,+\,B_{ij}^+ \sin\gamma_{ij}\theta\,,
      & \quad \phantom{0}\alpha_i<\theta<2\pi\,,
   \end{cases}
\edm
with appropriate values of $A_{ij}^\pm$ and $B_{ij}^\pm$.
We normalize these eigenfunctions in such a way that they define 
an orthonormal basis 
%of $L^2_{\alpha_i}(0,2\pi)$, which is the space
%of square integrable functions on the interval $(0,2\pi)$, equipped 
with respect to the inner product
\bdm
%\be{ip-ang}
%   \scalp{y,\widetilde y}_{L^2_{\alpha_i}(0,2\pi)}
   \scalp{y,\widetilde y}
   \,=\, \int_0^{\alpha_i} k\,y(\theta)\widetilde y(\theta)\dtheta
         \,+\, \int_{\alpha_i}^{2\pi} y(\theta)\widetilde y(\theta)\dtheta\,.
\edm
From the transmission conditions~\req{diff-eq-ang-transmission} it follows 
that
\be{linalg}
   \begin{cmatrix}
      1 & 0 & -\cos 2\pi\gamma & -\sin 2\pi\gamma \\
      0 & k & \phantom{-}\sin 2\pi\gamma & -\cos 2\pi\gamma \\
      \phantom{-k}\cos\gamma\alpha & \phantom{k}\sin\gamma\alpha 
            & -\cos\gamma\alpha & -\sin\gamma\alpha \\
      -k\sin\gamma\alpha & k\cos\gamma\alpha 
            & \phantom{-}\sin\gamma\alpha & -\cos\gamma\alpha
   \end{cmatrix}
   \begin{cmatrix} A^- \\ B^- \\ A^+ \\ B^+ \end{cmatrix}
   \,=\,
   \begin{cmatrix} 0 \\ 0 \\ 0 \\ 0 \end{cmatrix},
\ee
where we have omitted all subscripts $i$ and $j$ for the ease of readability;
in order that this system has a nontrivial solution 
%of this homogeneous linear system
the matrix $Y\in\R^{4\times 4}$ in \req{linalg} must be singular. 
In fact, we state the following result, which is implicit in Seo~\cite{Seo96}.

\begin{lemma}
\label{Lem:Y}
Let $Y=Y(\gamma,\alpha)$ be given by the $4\times 4$-matrix in \req{linalg}
with $\alpha\in(0,2\pi)\setminus\{\pi\}$ and $\gamma>0$.
Then $Y$ is singular, if and only if $\gamma$ is a solution of
\be{Lem:Y}
   \bigl|\sin(\gamma(\alpha-\pi))\bigr| \,=\, 
   \lambda\,\bigl|\sin\gamma\pi\bigr|
\ee
with $\lambda$ as in \req{gamma-cond}. If $\alpha$ and $\gamma$ satisfy
\req{Lem:Y} then $\operatorname{rank}(Y)=3$, unless $\gamma\in\N\setminus\{1\}$
and $\alpha=l\pi/\gamma$ for some 
$l\in\{1,\dots,2\gamma-1\}\setminus\{\gamma\}$,
in which case $\operatorname{rank}(Y)=2$ and the corresponding eigenvalue
problem~\req{diffeq-ang} has a two dimensional eigenspace.
\end{lemma}

\begin{proof}
Since we need this result in a slightly different context below 
(see \req{linalg-inhomogen0}) we provide a sketch of its proof. Introducing 
\bdm
   D \,=\, \begin{cmatrix} 1 & 0 \\ 0 & k \end{cmatrix}
   \qquad \text{and} \qquad
   R(\omega) \,=\, \begin{cmatrix}
                      \phantom{-}\cos\omega & \sin\omega \\
                      -\sin\omega & \cos\omega
                   \end{cmatrix}
\edm
we can write
\bdm
   Y \,=\,
   \begin{cmatrix} 
      D & -R(2\pi\gamma) \\ DR(\gamma\alpha) & -R(\gamma\alpha)
   \end{cmatrix},
\edm
and since its $(1,1)$-block $D$ is nonsingular, $Y$ is singular, 
if and only if the Schur complement
\bdmal
   S &\,=\, DR(\gamma\alpha)D^{-1}R(2\pi\gamma) - R(\gamma\alpha)\\[1ex]
     &\,=\, DR(\gamma\pi)
            \Bigl(R(\gamma(\alpha-\pi))D^{-1}R(\gamma\pi) 
                  \,-\, R(-\gamma\pi)D^{-1}R(\gamma(\alpha-\pi)\Bigr)
            R(\gamma\pi)
\edmal
of $D$ in $Y(\gamma,\alpha)$ is singular. 
Computing the products in the inner paranthesis and taking
the determinant of their difference it is readily seen that $S$
is singular, if and only if $\gamma$ is a solution of \req{Lem:Y}.
The Schur complement $S$ vanishes completely, if and only if
$\gamma\in\N$ and $\gamma\alpha$ 
is an integer multiple of $\pi$, in which case $y(\theta)=\cos\gamma\theta$
is one associated eigenfunction of \req{diffeq-ang}, 
and $y(\theta)=B^\pm\sin\gamma\theta$ with $B^+=kB^-$ is a second one.
\end{proof}

As mentioned in \cite{BFI92} the series~\req{BFI92} can be 
differentiated termwise to any order with respect to $\theta$ and $r$ 
(with one-sided derivatives at $\theta=\alpha_i,0,2\pi$, respectively). 
We thus compute
\begin{subequations}
\label{eq:normalderiv}
\begin{align}
\label{eq:normalderiv-a}
   \frac{\partial}{\partial\nu} u^-(x)
   &\,=\, -\frac{1}{r}\frac{\partial}{\partial\theta} u^-(x)
    \,=\, -\sum_{j=1}^\infty \beta_{ij}\gamma_{ij}B_{ij}^- r^{\gamma_{ij}-1}\,, 
    \qquad x\in\Gamma_{i+1}\,,
\intertext{and} 
\label{eq:normalderiv-b}
   \frac{\partial}{\partial\nu} u^-(x)
   &\,=\, \frac{1}{r}\frac{\partial}{\partial\theta} u^-(x)
    \,=\, \sum_{j=1}^\infty \beta_{ij}\gamma_{ij}
                      (B_{ij}^-c_{ij} - A_{ij}^-s_{ij})r^{\gamma_{ij}-1}\,, 
    \qquad x\in\Gamma_i\,,
\end{align}
\end{subequations}
for $x$ close to $x_i$, where we have set
%\be{cijsij}
\bdm
   c_{ij} = \cos\gamma_{ij}\alpha_i \qquad \text{and} \qquad
   s_{ij} = \sin\gamma_{ij}\alpha_i\,.
\edm
%\ee
Likewise, we have
\begin{subequations}
\label{eq:tanderiv}
\begin{align}
\label{eq:tanderiv-a}
   \frac{\partial}{\partial\tau} u(x)
   &\,=\, \frac{\partial}{\partial r} u(x)
    \,=\, \sum_{j=1}^\infty \beta_{ij}\gamma_{ij}A_{ij}^- r^{\gamma_{ij}-1}\,,
    \qquad x\in\Gamma_{i+1}\,,
\intertext{and} 
\label{eq:tanderiv-b}
   \frac{\partial}{\partial\tau} u(x)
   &\,=\, -\frac{\partial}{\partial r} u(x)
    \,=\, -\sum_{j=1}^\infty 
           \beta_{ij}\gamma_{ij}(A_{ij}^-c_{ij} + B_{ij}^-s_{ij}) r^{\gamma_{ij}-1}\,,
    \qquad x\in\Gamma_i\,,
\end{align}
\end{subequations}
near $x=x_i$.
%\begin{subequations}
%\label{eq:tanderiv}
%\begin{align}
%\label{eq:tanderiv-a}
%   \frac{\partial}{\partial\tau} u(x)
%   &\,=\, \frac{\partial}{\partial r} u(x)
%    \,=\, \sum_{j=1}^\infty \beta_{ij}\gamma_{ij}
%             (A_{ij}^+c_{ij}^+ + B_{ij}^+s_{ij}^+) r^{\gamma_{ij}-1}\,,
%    \qquad x\in\Gamma_{i+1}\,,
%\intertext{and} 
%\label{eq:tanderiv-b}
%   \frac{\partial}{\partial\tau} u(x)
%   &\,=\, -\frac{\partial}{\partial r} u(x)
%    \,=\, -\sum_{j=1}^\infty 
%  \beta_{ij}\gamma_{ij}(A_{ij}^+c_{ij} + B_{ij}^+s_{ij}) r^{\gamma_{ij}-1}\,,
%    \qquad x\in\Gamma_i\,,
%\end{align}
%\end{subequations}
%near $x=x_i$ with
%be{cijsijplus}
%   c_{ij}^+ \,=\, \cos 2\pi\gamma_{ij} \qquad \text{and} \qquad
%   s_{ij}^+ \,=\, \sin 2\pi\gamma_{ij}\,.
%\ee

From \req{gamma}, \req{normalderiv}, \req{tanderiv}, and \req{Bnu}, 
and from the fact that the gradient of $u$ has continuous extensions 
from either side to every edge of $\partial\D$ it follows that
\be{Heps}
   \frac{\partial}{\partial\nu} u^\pm\Big|_{\partial\D}\,, \ 
   \frac{\partial}{\partial\tau} u\Big|_{\partial\D}
   \,\in\,L^2(\partial\D)\,,
\ee
and, of course, the same is true for every solution $v_g$ of \req{vg-pde};
compare also Escauriaza, Fabes, and Verchota~\cite{EFV92}.
Accordingly, the right-hand side of \req{BFV17} is well-defined.
%We further note that $u$ belongs to the Sobolev space $H^{3/2-\eps}(\Omega)$
%for every $\eps>0$, cf., e.g., Theorem~1.2.18 in the book by 
%Grisvard~\cite{Gris92}.

%%%%%%%%%%%%%%%%%%%%%%%%%%%%%%%%%%%%%%%%%%%%%%%%%%%%%%%%%%%%%%%%%%%%%%%
\section{The shape derivative of \boldmath$u=u(\D)$ for a polygonal
inclusion in $\Omega$}
\label{Sec:estimates}
%%%%%%%%%%%%%%%%%%%%%%%%%%%%%%%%%%%%%%%%%%%%%%%%%%%%%%%%%%%%%%%%%%%%%%% 
In the sequel we investigate the existence of the shape derivative of
the solution $u=u(\D)$ of the conductivity equation~\req{conductivity-equation}
at a polygonal inclusion $\D$.
We also derive bounds for the corresponding shape derivative 
$\partial\Lambda_f(\D)$ and its associated Taylor remainder,
the latter improving upon the estimate~\req{Frechet-bound} 
provided in \cite{BFV17}.

Before doing so we first discuss how to extend a given vector field $h\in\S$ 
to a vector field
%We recall the assumption that a vector field $h\in\W$
\bdm
   H \,\in\, \W \,=\,
   \bigl\{\,H\in W^{1,\infty}(\Omega)\,:\,\supp H \subset\Omega\,\bigr\}
\edm
which is defined in all of $\Omega$ and satisfies $H|_{\partial\D}=h$. 
To this end we first choose a polygonal domain $\D_1$ with vertices 
$x_i^{(1)}\in\D$ in the vicinity of $x_i$, $i=1,\dots,n$, respectively,
such that $\overline\D_1\subset\D$ and each (open) quadrangle $\Q_i$ with 
vertices $x_i,x_{i+1},x_i^{(1)}$, and $x_{i+1}^{(1)}$ is convex; 
using the associated convex representation
\bdm
   x \,=\, c(1-d)\thsp x_i \,+\, cd\, x_{i+1} \,+\, 
           (1-c)(1-d)\thsp x_i^{(1)} \,+\, (1-c)d\, x_{i+1}^{(1)} 
\edm
with $0<c,d<1$ of a general point $x\in\Q_i$
we can extend $h$ to $\Q_i$ via
\bdm
   H(x) \,=\, c\, h\bigl((1-d)\thsp x_i \,+\, d\, x_{i+1}\bigr)\,, \qquad 
   x\in\Q_i\,,
\edm
and in a second step, extend $H$ continuously by zero to $\overline\D_1$.
In the same manner we can construct a polygonal domain $\D_2$ with
$\overline\D\subset\D_2$ and $\overline\D_2\subset\Omega$ and a corresponding
extension of $H$ to $\D_2$, and finally extend $H$ by zero to the rest
of $\Omega$. Note that $h\mapsto H$ is a bounded linear map, i.e., 
there exists a constant $\CT \geq 1$, independent of $h$, 
such that the extended vector field satisfies
\be{CT}
%   \cT \norm{h}_{W^{1,\infty}(\partial\D)} \,\leq\,
   \norm{H}_{W^{1,\infty}(\Omega)} \,\leq\, \CT\,\norm{h}_{W^{1,\infty}(\partial\D)}\,.
\ee
%In particular, $h_e$ satisfies \req{h-bound2} if
%\bdm
%   \norm{h}_{W^{1,\infty}(\partial\D)} \,\leq\, c\,:=\,\frac{1}{2C'\CT}\,.
%\edm

Given this extension $H$ of $h$ we
define
\be{Phi}
   \PhiG_{H}(x) \,=\, x \,+\, H(x)\,, \qquad x\in\Omega\,,
\ee
and note that $\PhiG_{H}:\Omega\to\Omega$ is bijective, if, for example,
\be{t0}
   \norm{H}_{W^{1,\infty}(\Omega)} 
%   \,=\, \sup_{x\in\Omega}\,\bigl(\norm{\red{H}(x)}_2
%         \,+\, \norm{\red{DH(x)}}_2\bigr)
   \,\leq\, 1/2\,,
\ee
by virtue of Banach's fixed point theorem.
Due to \req{CT} the condition~\req{t0} is satisfied for $h$
sufficiently small.

We now follow the standard argumentation from \cite{SoZo92}, which has been
worked out in \cite{BMPS18} for the present application. 
Let $H\in\W$ be defined as above and satisfy \req{t0}.
Denote by $u_h\in H^1_\diamond(\Omega)$ the weak solution of the 
forward problem
%\be{ut}
\bdm
   \nabla\cdot (\sigma_h \nabla u_h) \,=\, 0  \quad \text{in $\Omega$}\,, 
   \qquad
   \frac{\partial}{\partial\nu} u_h \,=\, f \quad \text{on $\partial\Omega$}\,,
\edm
%\ee
where
\bdm
   \sigma_h(x) \,=\, \begin{cases}
                        k\,, & \text{for $x\in\D_h$}\,,\\
                        1\,, & \text{for $x\in\Omega\setminus\overline\D_h$}\,.
                     \end{cases}
\edm
Note that $\D_h=\Phi_H(\D)$. Therefore, when introducing
%\be{utilde}
\bdm
   \utilde_{H} \,=\, u_h\circ\PhiG_{H} \,:\, \Omega\to\R\,,
\edm
it has been shown in \cite{BMPS18} that the \emph{material derivative}
\be{upunkt}
   \upunkt_{H} \,=\, \lim_{t\to 0} \frac{\utilde_{tH}-u}{t} 
   \,\in\, H^1_\diamond(\Omega)
\ee
is well-defined, the convergence being in $H^1(\Omega)$. Morover,
$\upunkt_{H}$ satisfies the variational problem
\be{upunkt-pde}
   \int_\Omega \sigma\,\nabla\upunkt_{H}\cdot\nabla w\dx
   \,=\, -\int_\Omega \sigma\,\nabla u\cdot ({\cal A}_{H}\nabla w)\dx \qquad
   \text{for every $w\in H^1(\Omega)$}\,,
\ee
where $\sigma$ is as in \req{sigma} and
%\be{A0}
\bdm
   {\cal A}_{H} \,=\, (\nabla\cdot H)I \,-\, DH \,-\, (DH)^T
\edm
-- with $DH$ being the Jacobi matrix of $H$ --
is an $L^\infty$ function of $2\times 2$ matrices with
\be{Ahbound}
   \norm{{\cal A}_{H}}_{L^\infty(\Omega)} 
   \,=\, \sup_{x\in\Omega} \norm{{\cal A}_{H}(x)}_2
   \,\leq\, 4\,\norm{H}_{W^{1,\infty}(\Omega)}\,.
\ee
It has further been proved in \cite{BMPS18} that 
$\upunkt_{H}|_{\partial\Omega}$ 
coincides with the shape derivative of $\Lambda_f(\D)$ in the direction
of $h=H|_{\partial\D}$, i.e.,
\be{BMPS18}
   \scalp{\upunkt_{H},g}_{L^2(\partial\Omega)} 
   \,=\, \scalp{\partial\Lambda_f(\D)h,g}_{L^2(\partial\Omega)}
\ee
is given by \req{BFV17} for every $g\in L^2_\diamond(\partial\Omega)$. 

%\color{red}
%\begin{remark}
%\label{Rem:h}
%\rm
%Note that $\D_h$, as introduced in Section~\ref{Sec:background}, 
%is defined for a vector field $h$ on $\partial\D$, 
%and so is the shape derivative $u_h'$,
%respectively $\partial\Lambda_f(\D)h$. The material derivative $\upunkt_h$,
%on the other hand, requires a vector field $h$, which is defined in all of
%$\Omega$.
%Below, we will therefore distinguish between a given vector field $h\in\S$
%-- when we talk about the shape derivative -- and a vector field
%$h\in \W$, when the material derivative is concerned.
%
%We mention that for the latter case we can always extend a given field $h\in\S$
%to some $h_e\in \W$ with the following construction.
%Then we construct an extension $h_e\in\W$ of this given function in the 
%following way. 
%\fin
%\end{remark}
%\color{black}

The approach from \cite{BMPS18} via the material derivative
allows for a stronger result than \req{Frechet-bound}.

\begin{theorem}
\label{Thm:Frechet-bound}
Let $\D$ be a polygon with simply connected closure $\overline\D\subset\Omega$.
Then the material derivative in ${\cal L}(\W,H_\diamond^1(\Omega))$
is a Fr\'echet derivative. Moreover, 
as far as the shape derivative $\partial\Lambda_f(\D)$ is concerned,
there are constants $c,C>0$, 
which only depend on $\Omega$, $\D$, $f$, and $k$, such that
\be{Thm:Frechet-bound}
   \norm{\Lambda_f(\D_h)-\Lambda_f(\D)-\partial\Lambda_f(\D)h}_{L^2(\partial\Omega)}
   \,\leq\, C\,\norm{h}_{W^{1,\infty}(\partial\D)}^2
\ee
for every $h\in\S$ with $\norm{h}_{W^{1,\infty}(\partial\D)}\leq c$.
% and for $t_0$ of \req{t0}.
\end{theorem}

\begin{proof}
Let $H\in\W$ satisfy \req{t0}.
%, and extend it to a function $h:\Omega\to\R^2$ as described above.
Then the Jacobian $D\PhiG_H$ of $\PhiG_H$ is invertible, 
% and \req{CT}. 
and it follows from \cite[Eq.~(2.11)]{BMPS18} that the auxiliary function 
$w_{H}=\utilde_{H}-u$ solves the variational problem
\be{BMPS18-2.11}
   \int_\Omega \sigma\, \nabla w_{H}\cdot (A_{H}\nabla w)\dx
   \,=\, \int_\Omega 
            \sigma\,\nabla u\cdot\bigl((I-A_{H})\nabla w\bigr)\dx
\ee
for every $w\in H^1(\Omega)$, where
\be{At}
   A_{H} \,=\, (D\PhiG_H)^{-1}(D\PhiG_H)^{-T}\det D\PhiG_H\,.
\ee
From \req{BMPS18-2.11} we deduce that
\bdmal
   \int_\Omega \sigma\,\nabla w_{H}\cdot\nabla w\dx
   &\,=\, \int_\Omega \sigma\,\nabla (\utilde_{H}-u)\cdot
                     \bigl((I-A_{H})\nabla w\bigr)\dx
          \,+\, \int_\Omega \sigma\,\nabla w_{H}\cdot(A_{H}\nabla w)\dx\\[1ex]
   &\,=\, \int_\Omega \sigma\,\nabla \utilde_{H}\cdot
                     \bigl((I-A_{H})\nabla w\bigr)\dx\,,
\edmal
and hence, \req{upunkt-pde} implies that the Taylor remainder
\bdm
   e_{H}\,=\, \utilde_{H}-u-\upunkt_{H} \,=\, w_{H}-\upunkt_{H}
\edm
of the material derivative satisfies
\be{Taylorrest}
\begin{aligned}
   &\!\!\!\int_\Omega \sigma\, \nabla e_{H}\cdot\nabla w\dx
    \,=\, \int_\Omega \sigma\,\nabla \utilde_{H}\cdot
                     \bigl((I-A_{H})\nabla w\bigr)\dx
          \,+\, \int_\Omega \sigma\,\nabla u\cdot({\cal A}_{H}\nabla w)\dx\\[1ex]
   &\qquad
    \,=\, \int_\Omega \sigma\,\nabla w_{H}\cdot \bigl((I-A_{H})\nabla w\bigr)\dx
          \,+\, \int_\Omega \sigma\,\nabla u \cdot
                   \bigl((I-A_{H}+{\cal A}_{H})\nabla w\bigr)\dx
\end{aligned}
\ee
for every $w\in H^1(\Omega)$. 

From \req{Phi} and \req{At} we have
\bdm
   A_{H} \,=\, (I+DH)^{-1}(I+(DH)^T)^{-1}(1+\nabla\cdot H+\det DH)\,,
\edm
and hence, using \req{t0}, it follows that there exists some constant $C'$,
which we take to be greater than one, such that
\bdm
   \norm{A_{H} - I}_{L^\infty(\Omega)} \,\leq\, C'\,\norm{H}_{W^{1,\infty}(\Omega)}
   \quad \text{and} \quad
   \norm{A_{H} - I - {\cal A}_{H}}_{L^\infty(\Omega)}
   \,\leq\, C'\,\norm{H}_{W^{1,\infty}(\Omega)}^2\,,
\edm
independent of the particular choice of $H$.
%, because $\norm{h}_{W^{1,\infty}(\Omega)}\leq\CT$ 
%under the given assumptions by virtue of \req{CT}.
Accordingly, the right-hand side of \req{Taylorrest} can be estimated by
\bdm
   C'\bigl(\norm{\sqrt\sigma\,\nabla w_{H}}_{L^2(\Omega)} 
           \norm{H}_{W^{1,\infty}(\Omega)}
           \,+\, \norm{\sqrt\sigma\,\nabla u}_{L^2(\Omega)}
                 \norm{H}_{W^{1,\infty}(\Omega)}^2\bigr)\,
   \norm{\sqrt\sigma\,\nabla w}_{L^2(\Omega)}\,,
\edm
and using $w=e_{H}$ in \req{Taylorrest} we thus obtain the estimate
\be{Taylorrest-estimate-tmp}
\begin{aligned}
   &\norm{\sqrt\sigma\,\nabla e_{H}}_{L^2(\Omega)} \,\leq\, \\[1ex]
   &\qquad 
    C'\bigl(\norm{\sqrt\sigma\,\nabla w_H}_{L^2(\Omega)} \norm{H}_{W^{1,\infty}(\Omega)}
            \,+\, \norm{\sqrt\sigma\,\nabla u}_{L^2(\Omega)}
                  \norm{H}_{W^{1,\infty}(\Omega)}^2\bigr)
\end{aligned}
\ee
for the Taylor remainder. 

Similarly, we conclude from \req{BMPS18-2.11} and \req{At} that
\bdm
   \frac{1}{2}\,\norm{\sqrt\sigma\,\nabla w_{H}}_{L^2(\Omega)}
   \,\leq\, C'\,\norm{\sqrt\sigma\,\nabla u}_{L^2(\Omega)} 
                \norm{H}_{W^{1,\infty}(\Omega)}
\edm
for every $H$ satisfying
\be{h-bound2}
   \norm{H}_{W^{1,\infty}(\Omega)} \,\leq\, \frac{1}{2C'}\,,
\ee
whereas the weak form \req{forwardvarproblem} of \req{conductivity-equation} 
implies that
\be{norm-u}
\begin{aligned}
   \norm{\sqrt\sigma\,\nabla u}_{L^2(\Omega)}^2
   &\,\leq\,\norm{f}_{L^2(\partial\Omega)} \norm{u}_{L^2(\partial\Omega)}\\[1ex]
   &\,\leq\,\frac{c_\Omega}{\min\{1,\sqrt{k}\}}\,
            \norm{f}_{L^2(\partial\Omega)}\norm{\sqrt\sigma\,\nabla u}_{L^2(\Omega)}\,,
\end{aligned}
\ee
where $c_\Omega$ depends on the norm of the trace operator from $H^1(\Omega)$ to
$L^2(\partial\Omega)$ and the Poincar\'e constant for $H^1_\diamond(\Omega)$.
With these two estimates and another use of the Poincar\'e inequality
it follows from \req{Taylorrest-estimate-tmp} that
\be{Taylorrest-estimate} 
   \norm{e_{H}}_{H^1(\Omega)} \,\leq\, C'' \norm{H}_{W^{1,\infty}(\Omega)}^2
\ee
with a constant $C''$ which is independent of $H$, 
as long as $H$ satisfies \req{h-bound2}. 
This shows that the material derivative in
${\cal L}(\W,H_\diamond^1(\Omega))$ is a Fr\'echet derivative.

Concerning the second claim~\req{Thm:Frechet-bound}
%Next, 
let $h\in\S$ be a given field on the boundary of $\D$, and
$H\in\W$ be an extension of $h$ as constructed above.
From \req{CT} it follows that $H$ satisfies \req{h-bound2}, if
\bdm
   \norm{h}_{W^{1,\infty}(\partial\D)} \,\leq\, c\,:=\,\frac{1}{2C'\CT}\,.
\edm
Using this extended vector field the trace of the associated
material derivative $\upunkt_{H}$ coincides with 
$\partial\Lambda_f(\D)h$,
and hence, the left hand side of \req{Thm:Frechet-bound} is the norm of
the trace on $\partial\Omega$ of
\bdm
   e_{H} \,=\, \utilde_{H} \,-\, u \,-\, \upunkt_{H}\,.
\edm
Accordingly, the desired inequality~\req{Thm:Frechet-bound} is a consequence 
of \req{Taylorrest-estimate}, \req{CT}, and the trace theorem.
\end{proof}

We also have the following estimate that is of independent interest.

\begin{theorem}
\label{Thm:derivative-bound}
Let $\D$ be a polygon with simply connected closure $\overline\D\subset\Omega$.
Then there is a constant $C>0$, depending only on $\Omega$, $f$, and $k$, 
such that
\bdm
   \norm{\partial\Lambda_f(\D')}_{{\cal L}(\Sp,L^2(\partial\Omega))} \,\leq\, C
\edm
%Fr\'echet derivative $\partial\Lambda_f$ is uniformly bounded 
%in ${\cal L}(\S,L^2(\partial\Omega))$ 
for all polygons $\D'$ with the same number of vertices, 
which are sufficiently close to those of $\D$.
\end{theorem}

\begin{proof}
For the given polygon $\D$ we choose $\D_1$ and $\D_2$ 
for the extension $H\in\W$ of any $h\in\S$ as described at the beginning 
of this section.
%the proof of Theorem~\ref{Thm:Frechet-bound}.
Then every polygon $\D'$ which is a sufficiently small perturbation of $\D$ as
specified in the statement of this theorem satisfies 
$\partial\D'\subset\D_2\setminus\overline\D_1$, and every vector field
$h\in W^{1,\infty}(\partial\D')$ can also be extended to a vector field 
$H\in\W$ supported in $\overline\D_2$ 
using the very same construction.
Moreover, the corresponding extensions satisfy
\be{CTp}
%   \cTp\,\norm{h}_{W^{1,\infty}(\partial\D')} \,\leq\, 
   \norm{H}_{W^{1,\infty}(\Omega)} \,\leq\, \CTp\norm{h}_{\Sp}
   \,\leq\, 2\CT\,\norm{h}_{\Sp}\,,
\ee
%with some constant $\CTp\leq 2\CT$,
provided $\D'$ is sufficiently close to $\D$.
In particular, if we stipulate that $\norm{h}_{\Sp}=1/(4\CT)$,
then $H$ satisfies \req{t0}.

It therefore follows from \req{BMPS18} and \req{upunkt-pde} with 
$w=\upunkt_{H}$ that
\bdmal
   \norm{\partial\Lambda_f(\D')h}_{L^2(\partial\Omega)}
   &\,\leq\, \frac{c_\Omega}{\min\{1,\sqrt{k}\}}\,
            \norm{\sqrt\sigma\,\nabla\upunkt_{H}}_{L^2(\Omega)} \\[1ex]
   &\,\leq\, \frac{c_\Omega}{\min\{1,\sqrt{k}\}}\,
             \norm{\sqrt\sigma\,\nabla u}_{L^2(\Omega)}
             \norm{{\cal A}_{H}}_{L^\infty(\Omega)}\,,
\edmal
where $c_\Omega$ is the same constant as in \req{norm-u}. 
From \req{norm-u}, \req{Ahbound}, and \req{CTp} we therefore deduce that
%\bdm
%   \norm{{\cal A}_{h_e}(x)}_2 \,\leq\, 4\,\norm{h'(x)}_2 \qquad
%   \text{for every $x\in\Omega$}\,,
%\edm
%and together with \req{norm-u} and \req{CTp} we therefore arrive at
\bdmal
   \norm{\partial\Lambda_f(\D')h}_{L^2(\partial\Omega)}
   &\,\leq\, \frac{4 c_\Omega^2}{\min\{1,k\}}\, 
             \norm{f}_{L^2(\Omega)}\,\norm{H}_{W^{1,\infty}(\Omega)}\\[1ex]
   &\,\leq\, \frac{8 c_\Omega^2\,\CT}{\min\{1,k\}}\, 
             \norm{f}_{L^2(\Omega)}\,\norm{h}_{W^{1,\infty}(\partial\D')}\,.
\edmal
This proves the assertion.
%shows that $\partial\Lambda_f$ is uniformly bounded for all polygons
%in the vicinity of $\D$ with the same number of vertices.
\end{proof}

\begin{remark}
\label{Rem:uniformTaylorremainder}
\rm
In the same way one can show that the constants $c$ and $C$ of
Theorem~\ref{Thm:Frechet-bound} can be chosen in such a way that the same
Taylor remainder estimate~\req{Thm:Frechet-bound} is valid for all
polygons $\D'$ sufficiently close to $\D$.
\fin
\end{remark}

Now we turn to the existence of the shape derivative $u'_h$ of $u$,
formally defined in \req{defdomderiv}. 

\begin{theorem}
\label{Thm:domderiv}
Let $\D$ be a polygon with simply connected closure $\overline\D\subset\Omega$. 
Then the solution $u=u(\D)$ of \req{conductivity-equation} has a
Fr\'echet shape derivative 
$\partial u(\D)\in {\cal L}(\S,L^2(\Omega))$.
The derivative $u'_h=\partial u(\D)h$ of $u$ in the direction of
$h\in\S$ is given by
\be{Thm:domderiv}
   u_h' \,=\, \upunkt_{H} \,-\, H\cdot\nabla u
\ee
for any extension $H\in \W$ of $h$ as described in the beginning of
this section.
The function $u_h'$ is harmonic in $\Omega\setminus\partial\D$, 
and it belongs to $H^1(\Omega\setminus\supp H)$. 
The Neumann boundary values of $u_h'$ on $\partial\Omega$ vanish, 
and the trace $u_h'|_{\partial\Omega}\in L^2_\diamond(\Omega)$ 
is the shape derivative $\partial\Lambda_f(\D)h$ 
of the given measurements on the boundary. 
\end{theorem}

\begin{proof}
The statement on the Fr\'echet derivative of $u=u(\D)$ and its
representation~\req{Thm:domderiv} follows with the aid of, e.g.,
\cite[Lemme~5.3.3]{HePi05} from the fact that the material derivative
is a Fr\'echet derivative 
and the extension operator $h\mapsto H$ is linear and bounded.

Let $h\in\S$ be fixed, and
consider any domain $\Omega'$ with $\overline\Omega'\subset\D$, 
and let $w\in C_0^\infty(\Omega')$. 
Then $\Omega'\subset\D_{th}$ for every $0<t<t'$ with $t'$ sufficiently small, 
and hence, both $u_{th}$ and $u$ are harmonic in $\Omega'$ for $0<t<t'$. 
Accordingly,
\bdm
   \int_{\Omega'} \frac{u_{th}-u}{t}\,\Delta w\dx \,=\, 0 \qquad 
   \text{for $0<t<t'$}\,,
\edm
and in the limit $t\to 0$ this yields
\bdm
   \int_{\Omega'} u'_h\Delta w\dx \,=\, 0 \qquad 
   \text{for every $w\in C_0^\infty(\Omega')$}\,.
\edm
It thus follows from Weil's lemma that $u_h'$ is 
harmonic in $\Omega'$, and hence in $\D$. The same argument applies for any 
domain $\Omega'$ with $\overline\Omega'\subset\Omega\setminus\overline\D$,
showing that $u_h'$ is also harmonic in $\Omega\setminus\overline\D$.

By virtue of \req{Thm:domderiv} $u_h'$ coincides with $\upunkt_{H}$ on
$\Omega\setminus\supp H$. 
Accordingly, $\upunkt_{H}$ is a harmonic function
in $\Omega\setminus\supp H$, $u_h'$ belongs to 
$H^1(\Omega\setminus\supp H)$, and there holds 
$u_h'|_{\partial\Omega}=\upunkt_{H}|_{\partial\Omega}
=\partial\Lambda_f(\D)h$, and
\bdm
   \frac{\partial}{\partial\nu} u_h' 
   \,=\, \frac{\partial}{\partial\nu} \upunkt_{H}
   \,\in\, H^{-1/2}(\partial\Omega)\,.
\edm
This Neumann derivative can be determined from the variational 
definition~\req{upunkt-pde} of $\upunkt_{H}$: Choosing an arbitrary test 
function $w\in C^\infty(\R^2)$, which vanishes on $\D\cup\supp H$, 
it follows from \req{upunkt-pde} that
\bdm
   0 \,=\, \int_{\Omega\setminus\overline\D} \nabla\upunkt_{H}\cdot\nabla w\dx 
     \,=\, \int_{\partial\Omega} w\,\frac{\partial}{\partial\nu}\upunkt_{H} \ds\,,
\edm
where the right-hand side is to be understood as the duality pairing of 
elements from $H^{\pm 1/2}(\partial\Omega)$.
Since the traces on $\partial\Omega$ of all these admissible test functions
are dense in $L^2(\partial\Omega)$, the Neumann derivative on 
$\partial\Omega$ of $\upunkt_{H}$, and hence of $u_h'$, must vanish.  
\end{proof}

%%%%%%%%%%%%%%%%%%%%%%%%%%%%%%%%%%%%%%%%%%%%%%%%%%%%%%%%%%%%%%%%%%%%%%%
\section{The transmission problem for polygonal inclusions}
\label{Sec:HeRu98}
%%%%%%%%%%%%%%%%%%%%%%%%%%%%%%%%%%%%%%%%%%%%%%%%%%%%%%%%%%%%%%%%%%%%%%% 
%We now investigate the transmission problem~\req{HeRu98} for a polygonal
%inclusion $\D$:
Given a simply connected polygon $\D$ 
as specified in Section~\ref{Sec:forward}
we now introduce the four-dimensional spaces $S_i$, $i=1,\dots,n$,
spanned by the singular functions
\be{uiprime}
   w_i(x) \,=\, \begin{cases}
                   \chi_i(x)\ytilde_i(\theta)r^{\gamma_{i1}-1}\,, & 
                      x\in\B_{r_i}(x_i)\,,\\
                   0\,, & x\in\Omega\setminus\overline{\B_{r_i}(x_i)}\,,
                \end{cases}
\ee
defined in terms of the local coordinate system~\req{localcoordinates}, with
\be{yitilde}
   \ytilde_i(\theta) \,=\, 
   \begin{cases}
      A_i'{}^-\cos(\gamma_{ij}-1)\theta \,+\, B_i'{}^-\sin(\gamma_{ij}-1)\theta\,,
      & \quad \phantom{\alpha_i}0<\theta<\alpha_i\,,\\[1ex]
      A_i'{}^+\cos(\gamma_{ij}-1)\theta \,+\, B_i'{}^+\sin(\gamma_{ij}-1)\theta\,,
      & \quad \phantom{0}\alpha_i<\theta<2\pi\,.
   \end{cases}
\ee
Then for every element
\be{V}
   v\,\in\,
   V \,:=\, H^1(\Omega\setminus\partial\D)
            \,\oplus\, S_1\oplus\,\dots\,\oplus\, S_n
\ee            
the restrictions $v^-=v|_\D$ and $v^+=v|_{\Omega\setminus\overline\D}$ 
have well defined traces on every edge $\Gamma_i$
of $\partial\D$. Such a trace is the sum of a function in 
$H^{1/2}(\Gamma_i)$ and a function in $C^\infty(\Gamma_i)$. 
Further, if $v^-$ and $v^+$ are harmonic functions,
say, then their normal derivatives on $\Gamma_i$ exist in
$H^{-1/2}(\Gamma_i)\oplus C^\infty(\Gamma_i)$. 
Corresponding properties are valid on $\partial\Omega$.

\begin{theorem}
\label{Thm:HeRu98-1}
Let $\D$ be as above and $h\in\S$.
%polygon with simply connected closure $\overline\D\subset\Omega$,
%and let $h\in\S$. 
Then the transmission problem 
\be{HeRu98-2}
\begin{array}{c}
   \Delta w \,=\, 0 \quad \text{in $\Omega\setminus\partial\D$}\,, \qquad
   \dfrac{\partial}{\partial\nu}w \,=\, 0 \quad \text{on $\partial\Omega$}\,,
   \qquad 
   {\displaystyle \int_{\partial\Omega} w\ds \,=\, 0\,,}
   \\[3ex]
   \bigl[w\bigr]_{\Gamma_i} 
   \,=\, (1-k)(h\cdot\nu)\dfrac{\partial}{\partial\nu}u^{-}\,, \quad
   \bigl[D_\nu w\bigr]_{\Gamma_i} 
   \,=\, (1-k)\,\dfrac{\partial}{\partial\tau}
                \Bigl((h\cdot\nu)\dfrac{\partial}{\partial\tau}u\Bigr)\,,
\end{array}
\ee
where $i$ runs from $1$ to $n$, has a unique solution $w\in V$, where $V$
is defined in \req{V}.
%, and this solution belongs to 
%$H^\gamma(\Omega\setminus\partial\D)$ for some $\gamma\in(1/2,1)$.
% and
%is $C^\infty$ smooth in $\Omega\setminus\partial\D$.
%and all its derivatives extend continuously to $\partial\D\setminus\{x_i\}$ 
%from either side. Finally, $w$ 
%also to $H^1\bigl(\Omega\setminus\overline{\D\cup_i\B_{r_i}(x_i)}\bigr)$.
\end{theorem}

\begin{proof}
Uniqueness is obvious: If 
$w_1,w_2\in V$ are two solutions of \req{HeRu98-2},
then $v=w_1-w_2$ solves the homogeneous Neumann problem~\req{vg-pde} with 
$g=0$. This proves that $v=0$.

To establish existence, let
\be{ai}
   h_i^- \,=\, \lim_{x\to x_i,\,x\in\Gamma_i} (h\cdot\nu)(x)\,, \qquad
   h_i^+ \,=\, \lim_{x\to x_i,\,x\in\Gamma_{i+1}} (h\cdot\nu)(x)\,,
\ee
for every $i=1,\dots,n$; take note that $h_i^-\neq h_i^+$ in general, and that
\be{h-local}
   (h\cdot\nu)(x) 
   \,=\, \begin{cases}
            h_i^-\,+\, O(|x-x_i|)\,, & x\in\Gamma_i\,,\\
            h_i^+\,+\, O(|x-x_i|)\,, & x\in\Gamma_{i+1}\,,
         \end{cases}
\ee
for $x$ near $x_i$.
Furthermore, for $i=1,\dots,n$ let $A_i'{}^\pm$ and $B_i'{}^\pm$ 
be the entries of the solution of the linear system
\be{linalg-inhomogen0}
   Y(\gamma_{i1}-1,\alpha_i) 
   \begin{cmatrix}
      A_i'{}^- \\ B_i'{}^- \\ A_i'{}^+ \\ B_i'{}^+
   \end{cmatrix}
   \,=\, (1-k)\gamma_{i1}
   \begin{cmatrix}
      h_i^+B_{i1}^- \\
      h_i^+A_{i1}^- \\
      \phantom{-}h_i^-(A_{i1}^-s_{i1}-B_{i1}^-c_{i1}) \\
      -h_i^-(A_{i1}^-c_{i1}+B_{i1}^-s_{i1})
   \end{cmatrix}
\ee
%\be{linalg-inhomogen0}
%   Y(\gamma_{i1}-1,\alpha_i) 
%   \begin{cmatrix}
%      A_i'{}^- \\ B_i'{}^- \\ A_i'{}^+ \\ B_i'{}^+
%   \end{cmatrix}
%   \,=\, (k-1)\gamma_{i1}
%   \begin{cmatrix}
%      h_i^+B_{i1}^- \\
%      \phantom{-}h_i^+(A_{i1}^+c_{i1}^++B_{i1}^+s_{i1}^+) \\
%      \phantom{-}h_i^-(A_{i1}^-s_{i1}-B_{i1}^-c_{i1}) \\
%      -h_i^-(A_{i1}^+c_{i1}+B_{i1}^+s_{i1})
%   \end{cmatrix}
%\ee
with $Y\in\R^{4\times 4}$ defined in \req{linalg}.
Since $\gamma'=\gamma_{i1}-1\in (-1/2,0)$ according to \req{gamma}, 
$\gamma'$ cannot be a solution of \req{gamma-cond}, and hence the
inhomogeneous linear system~\req{linalg-inhomogen0} has a unique solution 
by virtue of Lemma~\ref{Lem:Y}. 
The way this system is set up, the function 
$w_i\in S_i$ of \req{uiprime}, \req{yitilde} 
with these particular coefficients $A_i'{}^\pm$ and $B_i'{}^\pm$
%\be{uiprime}
%   w_i(x) \,=\, \begin{cases}
%                   \beta_{i1}\chi_i(x)\ytilde_i(\theta)r^{\gamma_{i1}-1}\,, & 
%                      x\in\B_{r_i}(x_i)\,,\\
%                   0\,, & x\in\Omega\setminus\overline{\B_{r_i}(x_i)}\,,
%                \end{cases}
%\ee
%defined in terms of the local coordinate system~\req{localcoordinates},
%with
%\be{yitilde}
%   \ytilde_i(\theta) \,=\, 
%   \begin{cases}
%      A_i'{}^-\cos(\gamma_{ij}-1)\theta \,+\, B_i'{}^-\sin(\gamma_{ij}-1)
%\theta\,,
%      & \quad \phantom{\alpha_i}0<\theta<\alpha_i\,,\\[1ex]
%      A_i'{}^+\cos(\gamma_{ij}-1)\theta \,+\, B_i'{}^+\sin(\gamma_{ij}-1)
%\theta\,,
%      & \quad \phantom{0}\alpha_i<\theta<2\pi\,,
%   \end{cases}
%\ee
is such that $\bigl[w_i\bigr]_{\partial\D}$ and 
$\bigl[D_\nu w_i\bigr]_{\partial\D}$
coincide near $x=x_i$ with the leading order terms of the prescribed
transmission data in \req{HeRu98-2}, compare \req{normalderiv} and
\req{tanderiv}. In fact, from \req{h-local} and \req{gamma} follows that
\begin{subequations}
\label{eq:u0prime-sources}
\begin{align}
\label{eq:phiprime}
   \varphi &\,:=\, (1-k)\,(h\cdot\nu)\frac{\partial}{\partial\nu} u^-
                  \,-\, \sum_{i=1}^n \bigl[w_i\bigr]_{\partial\D}
   \,\in\, H^{1/2}(\partial\D)
\intertext{and}
\label{eq:psiprime}
   \psi &\,:=\, (1-k)\,\dfrac{\partial}{\partial\tau}
                      \Bigl((h\cdot\nu)\dfrac{\partial}{\partial\tau} u\Bigr)
               \,-\, \sum_{i=1}^n \bigl[D_\nu w_i\bigr]_{\partial\D} 
   \,\in\, H^{-1/2}(\partial\D)\,.
\intertext{Further, since $\ytilde_ir^{\gamma_{i1}-1}$ is harmonic in
$\supp\chi_i\setminus\partial\D$,}
\label{eq:Fprime}
   F_i &\,=\, \sigma\Delta w_i
         \,=\, \sigma\beta_{i1}
                \bigl(2\nabla\chi_i\cdot\nabla(\ytilde_ir^{\gamma_{i1}-1})
                      \,+\, \ytilde_ir^{\gamma_{i1}-1}\Delta\chi_i\bigr)\,, 
\end{align}
\end{subequations}
defined in $\Omega\setminus\partial\D$, is smooth in both subdomains
and supported in a small annulus around $x_i$. In particular, 
$F_i\in L^2(\Omega)$, and hence the transmission problem
\be{u0prime}
\begin{array}{c}
   -\nabla\cdot(\sigma\nabla w_0) \,=\,
   {\displaystyle \sum_{i=1}^n F_i} \qquad 
   \text{in $\Omega\setminus\partial\D$}\,, \\[3ex]
   \dfrac{\partial}{\partial\nu} w_0 = 0 \quad \text{on $\partial\Omega$}\,, 
   \qquad
   {\displaystyle \int_{\partial\Omega} w_0\ds \,=\, 0}\,, \\[4ex]
   \bigl[w_0\bigr]_{\partial\D} 
   \,=\, \varphi\,,
   \qquad
   \bigl[D_\nu w_0\bigr]_{\partial\D}
   \,=\, \psi\,,
\end{array}
\ee 
has a unique solution $w_0\in H^1(\Omega\setminus\partial\D)$, 
if and only if the integrability condition
\be{Integrabilitaetsbedingung}
   \sum_{i=1}^n \int_{\Omega\setminus\partial\D} F_i\dx 
   \,-\, \int_{\partial\D} \psi\ds \,=\, 0
\ee
is satisfied: compare Costabel and Stephan~\cite{CoSt85}. 
%In this case the restrictions of $w_0$ to $\D$ and
%$\Omega\setminus\overline\D$ belong to $H^1(\D)$ and
%$H^1(\Omega\setminus\overline\D)$, respectively

In order to check \req{Integrabilitaetsbedingung} one has to be careful, 
because although the two integrals in \req{Integrabilitaetsbedingung} are 
well-defined, the two individual terms of $\psi$ in \req{psiprime}
fail to be integrable in a classical sense. 
To take this into account we choose $\delta>0$ so small that 
$\chi_i$ is equal to one in $\B_\delta(x_i)$ for every $i=1,\dots,n$,
in which case
\bdm
   \frac{\partial}{\partial\nu} w_i(x) 
   \,=\, \beta_{i1}(\gamma_{i1}-1)\,\ytilde_i(\theta)\,\delta^{\gamma_{i1}-2} \qquad
   \text{on $\partial\B_\delta(x_i)$}\,,
\edm
where $\nu$ denotes the exterior normal vector on $\partial\B_\delta(x_i)$.
Let $\B_\delta$ be the union of these disks $\B_\delta(x_i)$, $i=1,\dots,n$.
%and denote by $\nu$ the exterior normal vector on $\partial\B_\delta$. 
Then Green's formula, applied in $\D\setminus\overline\B_\delta$ and in
$\Omega\setminus\overline{\D\cup\B_\delta}$, yields
\bdmal
   &\int_{\Omega\setminus\partial\D} F_i\dx
    \,=\, \int_{\Omega\setminus\overline{\B_\delta\cup\partial\D}} F_i\dx
    \,=\, \int_{\Omega\setminus\overline{\D\cup\B_\delta}} \Delta w_i\dx
          \,+\, \int_{\D\setminus\overline\B_\delta} k\,\Delta w_i\dx
          \\[1ex]
   &\qquad
    \,=\, \int_{\partial\Omega} \frac{\partial}{\partial\nu} w_i\ds
          \,-\, \int_{\partial\D\setminus\overline\B_\delta} 
                   \bigl[D_\nu w_i\bigr]_{\partial\D}\ds
          \,-\, \int_{\partial\B_\delta} 
                   \sigma\frac{\partial}{\partial\nu} w_i\ds\\[1ex]
   &\qquad
    \,=\, -\int_{\partial\D\setminus\overline\B_\delta} 
             \bigl[D_\nu w_i\bigr]_{\partial\D}\ds
          \,-\, \beta_{i1}(\gamma_{i1}-1)\,\delta^{\gamma_{i1}-1}
                \Bigl(
                   \int_0^{\alpha_i} k\,\ytilde_i(\theta)\dtheta
                   \,+\, \int_{\alpha_i}^{2\pi} \ytilde_i(\theta)\dtheta\Bigr)
\edmal
for every $i=1,\dots,n$. From \req{psiprime} we further have
\bdmal
   \int_{\partial\D\setminus\overline\B_\delta} \psi\ds 
   &\,=\, (1-k)\sum_{i=1}^n 
               \int_{\Gamma_i\setminus\overline\B_\delta} 
                  \frac{\partial}{\partial\tau}
                  \Bigl((h\cdot\nu)\frac{\partial}{\partial\tau} u\Bigr)\ds
          \,-\, \sum_{i=1}^n\int_{\partial\D\setminus\overline\B_\delta} 
                              \bigl[D_\nu w_i\bigr]_{\partial\D}\ds\\[1ex]
   &\,=\, (1-k)\sum_{i=1}^n 
                  \Bigl[(h\cdot\nu)\frac{\partial}{\partial\tau} u\Bigr]_i
%{\partial(\Gamma_i\setminus\overline\B_\delta)}
          \!-\, \sum_{i=1}^n \int_{\partial\D\setminus\overline\B_\delta} 
                               \bigl[D_\nu w_i\bigr]_{\partial\D}\ds\,,
\edmal
where
\be{eckig-i}
   \Bigl[(h\cdot\nu)\frac{\partial}{\partial\tau} u\Bigr]_i
   \,=\, -\beta_{i1}\gamma_{i1}
          \bigl(h_i^-(A_{i1}^-c_{i1}+B_{i1}^-s_{i1}) \,+\, h_i^+A_{i1}^-\bigr)
          \delta^{\gamma_{i1}-1} \,+\, o(1)
\ee
as $\delta\to 0$ is a short-hand notation for the difference of the values of 
$(h\cdot\nu)\frac{\partial}{\partial\tau}u$ at 
$\Gamma_i\cap\partial\B_\delta(x_i)$ and $\Gamma_{i+1}\cap\partial\B_\delta(x_i)$, 
which amounts to the right-hand side of
\req{eckig-i} according to \req{tanderiv} and \req{h-local}.
Since \req{yitilde} and the second and fourth equations
of \req{linalg-inhomogen0} give
\be{integral-identity}
\begin{aligned}
   &(\gamma_{i1}-1)\Bigl(\int_0^{\alpha_i} k\,\ytilde_i(\theta)\dtheta
                        \,+\, \int_{\alpha_i}^{2\pi} \ytilde_i(\theta)\dtheta
                  \Bigr)\\[1ex]
   &\qquad \qquad
    \,=\, (1-k)\gamma_{i1}
         \bigl(h_i^+A_{i1}^- \,+\,h_i^-(A_{i1}^-c_{i1}+B_{i1}^-s_{i1})\bigr)\,,
\end{aligned}
\ee
it follows that
%The sum of all these equations thus yields
\bdmal
   &\sum_{i=1}^n \int_{\Omega\setminus\partial\D} F_i\dx 
    \,-\, \int_{\partial\D\setminus\overline\B_\delta} \psi\ds \\[1ex]
   &\ 
    \,=\, (k-1)\sum_{i=1}^n 
                  \Bigl[(h\cdot\nu)\frac{\partial}{\partial\tau} u\Bigr]_i
%          \\[1ex]
%               \right|_{\partial(\Gamma_i\setminus\overline\B_\delta)}\\[1ex]
%   &\qquad\qquad\qquad\quad
          \,-\, \sum_{i=1}^n \beta_{i1} (\gamma_{i1}-1)\,\delta^{\gamma_{i1}-1}
                \Bigl(
                   \int_0^{\alpha_i} k\,\ytilde_i(\theta)\dtheta
                   \,+\, \int_{\alpha_i}^{2\pi} \ytilde_i(\theta)\dtheta\Bigr)
\edmal
converges to zero as $\delta\to 0$.
%\be{Integrbed-tmp}
%\begin{aligned}
%   &\sum_{i=1}^n \int_{\Omega\setminus\partial\D} F_i\dx 
%    \,+\, \int_{\partial\D\setminus\overline\B_\delta} \psi\ds
%    \,=\, (k-1)\sum_{i=1}^n 
%                  \Bigl[(h\cdot\nu)\frac{\partial}{\partial\tau} u\Bigr]_i
%          \\[1ex]
%%               \right|_{\partial(\Gamma_i\setminus\overline\B_\delta)}\\[1ex]
%   &\qquad\qquad\qquad\quad
%         \,-\, \sum_{i=1}^n \beta_{i1} (\gamma_{i1}-1)\,\delta^{\gamma_{i1}-1}
%                \Bigl(
%                   \int_0^{\alpha_i} k\,\ytilde_i(\theta)\dtheta
%                   \,+\, \int_{\alpha_i}^{2\pi} \ytilde_i(\theta)\dtheta\Bigr)
%          \\[1ex]
%   &\longrightarrow\, 0 \qquad \text{as $\delta\to 0$}\,.
%\end{aligned}
%\ee

We thus have shown that \req{u0prime} has a solution 
$w_0\in H^1(\Omega\setminus\partial\D)$, and hence,
\be{uprime-split}
   w \,=\, w_0 \,+\, \sum_{i=1}^n w_i \,\in\, V
\ee
solves \req{HeRu98-2}. 
\end{proof}

\begin{remark}
\rm
Since $w_i$ belongs to 
%$H^{1/2-\eps}(\Omega)$ for every $\eps>0$ and to $H^\gamma(\D)$ and 
$H^\gamma(\Omega\setminus\partial\D)$ for every
\bdm
   \gamma \,<\, \min\{\gamma_{i1}\,:\, i=1,\dots,n\}\,,
\edm
cf., e.g., \cite[Theorem~1.2.18]{Gris92}, this is also true for $w$. 
%Furthermore, as $w$ is harmonic in $\D$ and in 
%$\Omega\setminus\overline\D$, $w$ and $w_0$ are $C^\infty$ smooth in 
%$\D$ and in $\Omega\setminus\overline\D$. In fact, by virtue of 
%\cite[Theorem~4.20]{McLe00}
%all derivatives of $w_0$ extend continuously to $\partial\D\setminus\{x_i\}$ 
%from either side, and hence, so do those of $w$.
\fin
\end{remark}

%%%%%%%%%%%%%%%%%%%%%%%%%%%%%%%%%%%%%%%%%%%%%%%%%%%%%%%%%%%%%%%%%%%%%%%
\section{The shape derivative of \boldmath $u$ solves 
the transmission problem}
\label{Sec:domderiv}
%%%%%%%%%%%%%%%%%%%%%%%%%%%%%%%%%%%%%%%%%%%%%%%%%%%%%%%%%%%%%%%%%%%%%%% 
In Section~\ref{Sec:estimates} we have demonstrated that $u=u(\D)$ admits a 
well-defined shape derivative $u_h'=\partial u(\D)h$ in the direction of
any given vector field $h\in\S$. 
In the sequel it will be shown that $u_h'$ coincides with the solution $w$ 
of the transmission problem~\req{HeRu98-2}.

\begin{theorem}
\label{Thm:domderivPDE}
Let $\D$ be a polygon with simply connected closure $\overline\D\subset\Omega$.
Then the shape derivative of $u=u(\D)$ in direction
$h\in \S$ is the solution 
of the transmission problem~\req{HeRu98-2}.
\end{theorem}

\begin{proof}
For a given vector field $h\in\S$
let $w$ denote the solution of the transmission problem~\req{HeRu98-2}.
%where the vector field which enters the inhomogeneous transmission data
%in \req{HeRu98-2} is the trace of $h\in\W$ on $\partial\D$. 
In the first and major step of this proof
we show that the boundary values of $w$ on $\partial\Omega$ coincide with 
$u_h'|_{\partial\Omega}=\partial\Lambda_f(\D)h$, 
compare Theorem~\ref{Thm:domderiv}.

Since $w$ belongs to the space $V$ of \req{V} 
%$H^1\bigl(\Omega\setminus\overline{\D\cup_i\B_{r_i}(x_i)}\bigr)$
by virtue of Theorem~\ref{Thm:HeRu98-1}, its trace on $\partial\Omega$
is well-defined in $H^{1/2}(\partial\Omega)\subset L^2(\partial\Omega)$,
and it has vanishing mean according to \req{HeRu98-2}.
In view of \req{BFV17} we need to show that
\be{Thm:BFV17}
   \scalp{w,g}_{L^2(\partial\Omega)}
   \,=\, (1-k) \int_{\partial\D} (h\cdot\nu)\,\nabla u^- \cdot(M\nabla v_g^-)\ds
\ee
for every $g\in L^2_\diamond(\partial\Omega)$, where $v_g$ is the solution 
of \req{vg-pde} and $M\in\R^{2\times 2}$ is the symmetric $2\times 2$-matrix 
with eigenvalues $1$ and $k$ and eigenvectors $\tau$ and $\nu$, respectively.

As in the proof of Theorem~\ref{Thm:HeRu98-1} let $\delta>0$ be so small
that $\chi_i(x)=1$ for every $x\in\B_\delta(x_i)$ and every $i=1,\dots,n$,
set $\B_\delta=\cup_i\B_\delta(x_i)$, and denote by $\nu$ the exterior normal
vector on $\partial\B_\delta$. 
Further, let $w^-=w|_\D$ and $w^+=w|_{\Omega\setminus\overline\D}$. 
Since the Neumann derivative of $w$ vanishes
on $\partial\Omega$ by virtue of \req{HeRu98-2}, Green's formula in
$\Omega\setminus\overline{\D\cup\B_\delta}$ and in 
$\D\setminus\overline\B_\delta$ yields
\bdmal
   &\scalp{w,g}_{L^2(\partial\Omega)}
    \,=\, \int_{\partial\Omega} w\,\frac{\partial}{\partial\nu}v_g \ds
          \,-\, \int_{\partial\Omega} v_g\,\frac{\partial}{\partial\nu}w \ds
          \\[1ex]
   &\qquad
    \,=\, -\int_{\partial\D\setminus\B_\delta} v_g\bigl[D_\nu w\bigr]_{\partial\D} \ds
          \,+\, \int_{\partial\D\setminus\B_\delta} 
                   w{}^+\,\frac{\partial}{\partial\nu}v_g^+\ds
          \,-\, \int_{\partial\D\setminus\B_\delta}
                   kw{}^-\,\frac{\partial}{\partial\nu}v_g^-\ds \\[1ex]
   &\qquad\phantom{\,=\,}
          \,-\, \int_{\partial\B_\delta} 
                   \sigma v_g \,\frac{\partial}{\partial\nu}w\ds
          \,+\, \int_{\partial\B_\delta} 
                   \sigma w \,\frac{\partial}{\partial\nu}v_g\ds\,.
\edmal
Using 
\bdm
   w^-|_{\partial\D} \,=\, w^+|_{\partial\D} \,-\, \bigl[w\bigr]_{\partial\D}
\edm
and the transmission conditions~\req{HeRu98-2} as well as the fact that
$\bigl[D_\nu v_g\bigr]_{\partial\D}=0$ we can transform this equation further into
\bdmal
   &\scalp{w,g}_{L^2(\partial\Omega)}
    \,=\, (k-1)\int_{\partial\D\setminus\B_\delta}
                  v_g\,\frac{\partial}{\partial\tau}
                  \Bigl((h\cdot\nu)\dfrac{\partial}{\partial\tau} u\Bigr) 
               \ds\\[1ex]
   &\quad \,
          \,+\, k(1-k)\int_{\partial\D\setminus\B_\delta}
                         (h\cdot\nu)\dfrac{\partial}{\partial\nu} u^-
                         \,\frac{\partial}{\partial\nu}v_g^-\ds 
          \,-\, \int_{\partial\B_\delta}
                   \sigma v_g \,\frac{\partial}{\partial\nu}w\ds
          \,+\, \int_{\partial\B_\delta} 
                   \sigma w \,\frac{\partial}{\partial\nu}v_g\ds\,.
\edmal
The first term on the right-hand side consists of $n$ integrals over
the individual edges of $\D$, and partial integration on each
of these edges gives
\bdmal
   &\scalp{w,g}_{L^2(\partial\Omega)}
    \,=\, (k-1)\sum_{i=1}^n 
                  \Bigl[(h\cdot\nu)\thsp v_g\frac{\partial}{\partial\tau} u
                  \Bigr]_i
          +\, (1-k)\int_{\partial\D\setminus\B_\delta} \!
                        (h\cdot\nu)\frac{\partial}{\partial\tau} u\,
                        \frac{\partial}{\partial\tau} v_g \ds
          %         \,-\, k(k-1)\int_{\partial\D\setminus\B_\delta}\!
%                   h\cdot\nu\,\dfrac{\partial}{\partial\nu} u^-
%                      \,\frac{\partial}{\partial\nu}v_g^-\ds 
\\[1ex]
   &\quad \ \ 
          \,+\, k(1-k)\int_{\partial\D\setminus\B_\delta}\!
                         (h\cdot\nu)\dfrac{\partial}{\partial\nu} u^-
                         \,\frac{\partial}{\partial\nu}v_g^-\ds 
%          \,+\, (1-k)\sum_{i=1}^n 
%                        v_gh\cdot\nu\,\frac{\partial}{\partial\tau} u
%                        \Big|_{\partial(\Gamma_i\setminus\B_\delta)}
          \,-\, \int_{\partial\B_\delta}\!
                   \sigma v_g \,\frac{\partial}{\partial\nu}w\ds
          \,+\, \int_{\partial\B_\delta} \!
                   \sigma w \,\frac{\partial}{\partial\nu}v_g\ds\,,
\edmal
where we utilize the notation introduced in \req{eckig-i}.
Note that this identity can be rewritten as
\be{Frechet-tmp}
\begin{aligned}
   &\scalp{w,g}_{L^2(\partial\Omega)}
    \,=\, (1-k) \int_{\partial\D\setminus\B_\delta} 
                   (h\cdot\nu)\,\nabla u^- \cdot(M\nabla v_g^-)\ds\\[1ex]
   &\ 
          \,+\, (k-1)\sum_{i=1}^n 
                        \Bigl[
                           (h\cdot\nu)\thsp v_g\frac{\partial}{\partial\tau} u
                        \Bigr]_i
          \,-\, \int_{\partial\B_\delta}
                   \sigma v_g \,\frac{\partial}{\partial\nu}w\ds
          \,+\, \int_{\partial\B_\delta} 
                   \sigma w \,\frac{\partial}{\partial\nu}v_g\ds
\end{aligned}
\ee
with $M$ as in \req{BFV17}, 
and the first integral converges to the right-hand side of \req{Thm:BFV17}
as $\delta\to 0$ by virtue of \req{Heps}.
It therefore remains to show that the sum of terms in the second line of
\req{Frechet-tmp} converges to zero as $\delta\to 0$.

We now investigate these three summands individually. For the first one
we utilize \req{BFI92}, \req{tanderiv}, and \req{ai} to obtain
\begin{subequations}
\label{eq:Frechet-tmp-pieces}
\begin{align}
\nonumber
   &\sum_{i=1}^n \Bigl[(h\cdot\nu)\thsp v_g\frac{\partial}{\partial\tau} u
                \Bigr]_i\\[1ex]
\nonumber
   &\ =\, \sum_{i=1}^n 
             \bigl(h_i^-+O(\delta)\bigr)
             \bigl(v_g(x_i)+O(\delta^{\gamma_{i1}})\bigr)
             \bigl(-\beta_{i1}\gamma_{i1}(A_{i1}^-c_{i1}+B_{i1}^-s_{i1})
                   \delta^{\gamma_{i1}-1}+O(\delta^{\gamma_{i2}-1})\bigr)\\
\nonumber
   &\ \phantom{=\,}
          \,-\,
          \sum_{i=1}^n 
             \bigl(h_i^++O(\delta)\bigr)
             \bigl(v_g(x_i)+O(\delta^{\gamma_{i1}})\bigr)
             \bigl(\beta_{i1}\gamma_{i1}A_{i1}^-
                   \delta^{\gamma_{i1}-1}
                   +O(\delta^{\gamma_{i2}-1})\bigr)\\[1ex]
\label{eq:Frechet-tmp-pieces1}
   &\ =\, \displaystyle{
          -\sum_{i=1}^n 
              \beta_{i1}\gamma_{i1} v_g(x_i)
              \bigl(h_i^-(A_{i1}^-c_{i1}+B_{i1}^-s_{i1})
                    \,+\, h_i^+A_{i1}^-\bigr)
              \delta^{\gamma_{i1}-1}} \,+\, o(1)
\end{align}
as $\delta\to 0$, where the estimate of the remainder follows from
\req{gamma}.

Concerning the two integrals over the boundary of $\B_\delta$ in 
\req{Frechet-tmp} we decompose $w$ as in \req{uprime-split} and investigate
the corresponding parts separately. Consider $w_0$ first, i.e., the solution 
of the transmission problem~\req{u0prime}, and let
$w_0^-=w_0|_\D$ and $w_0^+=w_0|_{\Omega\setminus\overline\D}$.
Since the source terms $F_i$ of \req{u0prime} vanish in $\B_\delta$, 
compare~\req{Fprime},
$w_0$ and $v_g$ are both harmonic in $\B_\delta\setminus\partial\D$. 
Since both functions also belong to $H^1(\B_\delta\setminus\partial\D)$
we can apply Green's formula 
in $\B_\delta\cap\D$ and in $\B_\delta\setminus\overline\D$ to obtain
\bdmal
   &\int_{\partial\B_\delta} 
       \sigma w_0 \,\frac{\partial}{\partial\nu}v_g\ds
    \,-\, \int_{\partial\B_\delta} 
            \sigma v_g \,\frac{\partial}{\partial\nu}w_0\ds \\[1ex]
   &\qquad
    \,=\, -\int_{\partial\D\cap\B_\delta} v_g\bigl[D_\nu w_0\bigr]_{\partial\D}\ds
          \,-\, \int_{\partial\D\cap\B_\delta}
                   k\,w_0^-\,\frac{\partial}{\partial\nu}v_g^-\ds
          \,+\, \int_{\partial\D\cap\B_\delta}
                   w_0^+\,\frac{\partial}{\partial\nu}v_g^+\ds\,.
\edmal
Using \req{u0prime} and the fact that 
\bdm
   w_0^+|_{\partial\D} \,=\, w_0^-|_{\partial\D} \,+\, \bigl[w_0\bigr]_{\partial\D}\,,
\edm
we thus arrive at
\bdmal
   &\int_{\partial\B_\delta} 
       \sigma w_0 \,\frac{\partial}{\partial\nu}v_g\ds
    \,-\, \int_{\partial\B_\delta} 
             \sigma v_g \,\frac{\partial}{\partial\nu}w_0\ds \\[1ex]
   &\quad
    \,=\, -\int_{\partial\D\cap\B_\delta} v_g\bigl[D_\nu w_0\bigr]_{\partial\D}\ds
          \,+\, \int_{\partial\D\cap\B_\delta}
                   w_0^-\bigl[D_\nu v_g\bigr]_{\partial\D}\ds
          \,+\, \int_{\partial\D\cap\B_\delta}
                   \bigl[w_0\bigr]_{\partial\D}
                   \frac{\partial}{\partial\nu}v_g^+\ds\\[1ex]
   &\quad
    \,=\, -\int_{\partial\D\cap\B_\delta} v_g\psi\ds
          \,+\, \int_{\partial\D\cap\B_\delta}
                   \varphi\,\frac{\partial}{\partial\nu}v_g^+\ds\,.
\edmal
The two products $v_g\psi$ and $\varphi\frac{\partial}{\partial\nu}v_g^+$
are integrable over $\partial\D$
because $\psi\in H^{-1/2}(\partial\D)$ and $\varphi\in H^{1/2}(\partial\D)$, 
cf.~\req{u0prime-sources}, and from this we conclude that
\be{Frechet-tmp-pieces2}
   \int_{\partial\B_\delta} 
      \sigma w_0 \,\frac{\partial}{\partial\nu}v_g\ds
   \,-\, \int_{\partial\B_\delta} 
            \sigma v_g \,\frac{\partial}{\partial\nu}w_0\ds
   \,=\, o(1)\,, \qquad \delta\to 0\,.
\ee

For every fixed $i=1,\dots,n$ we have 
$\supp w_i\cap\B_\delta \subset \B_\delta(x_i)$, and hence,
\bdm
   \int_{\partial\B_\delta} 
      \sigma w_i \,\frac{\partial}{\partial\nu}v_g\ds
   \,=\, \int_{\partial\B_\delta(x_i)} 
            \sigma w_i \,\frac{\partial}{\partial\nu}v_g\ds\,.
\edm
%for any fixed $i\in\{1,\dots,n\}$. 
It further follows from \req{uiprime} and \req{BFI92} that
\bdm
   |w_i(x)| \,=\, O(\delta^{\gamma_{i1}-1}) 
\edm
and
\bdm
   \left|\frac{\partial}{\partial\nu}v_g(x)\right| \,=\, 
   \left|\frac{\partial}{\partial r}v_g(x)\right| \,=\, O(\delta^{\gamma_{i1}-1})
\edm
on $\partial\B_\delta(x_i)$. Inserting these estimates into the 
%integral in question shows 
above integral we conclude that
\be{Frechet-tmp-pieces3}
   \int_{\partial\B_\delta} 
      \sigma w_i \,\frac{\partial}{\partial\nu}v_g\ds
   \,=\, O(\delta^{2\gamma_{i1}-1}) \,=\, o(1)\,, \qquad \delta\to 0\,,
\ee
by virtue of \req{gamma}.

Finally, since $\delta$ is so small that $\chi_i=1$ in $\B_\delta(x_i)$
the Neumann boundary derivative of $w_i$ satisfies
\bdm
   \frac{\partial}{\partial\nu}w_i
   \,=\, \begin{cases}
            \beta_{i1}(\gamma_{i1}-1)\ytilde_i(\theta)\,\delta^{\gamma_{i1}-2}\,, & 
                   \text{on $\partial\B_\delta(x_i)$}\,, \\[1ex]
            0\,, & \text{on $\partial\B_\delta(x_j)$, $j\neq i$}\,,
         \end{cases}
\edm
while
\bdm
   v_g \,=\, v_g(x_i) \,+\, O(\delta^{\gamma_{i1}}) \qquad 
   \text{on $\partial\B_\delta(x_i)$}\,,
\edm
and hence,
\be{Frechet-tmp-pieces4}
\begin{aligned}
   &\int_{\partial\B_\delta} \!\!
       \sigma v_g \,\frac{\partial}{\partial\nu}w_i\ds
    \,=\, \beta_{i1}(\gamma_{i1}-1)v_g(x_i)
          \left(\int_0^{\alpha_i} k\,\ytilde_i(\theta)\dtheta
                + \int_{\alpha_i}^{2\pi} \ytilde_i(\theta)\dtheta\right)
          \delta^{\gamma_{i1}-1}\\[1ex]
   &\qquad \qquad \qquad \qquad \ 
    \,+\, O(\delta^{2\gamma_{i1}-1})\,.
\end{aligned}
\ee
\end{subequations}

Assembling the pieces of \req{Frechet-tmp-pieces} and inserting them into
\req{Frechet-tmp} we thus obtain
\bdmal
   &\scalp{w,g}_{L^2(\partial\Omega)}
    \,=\, (1-k) \int_{\partial\D} 
                   (h\cdot\nu)\,\nabla u^- \cdot(M\nabla v_g^-)\ds\\[1ex]
   &\qquad
          \,-\,(k-1)\sum_{i=1}^n 
               \beta_{i1}\gamma_{i1} v_g(x_i)
               \bigl(h_i^-(A_{i1}^-c_{i1}+B_{i1}^-s_{i1}) \,+\, h_i^+A_{i1}^-\bigr)
               \delta^{\gamma_{i1}-1} \\[1ex]
   &\qquad
          \,-\,\sum_{i=1}^n 
               \beta_{i1}(\gamma_{i1}-1)v_g(x_i)
               \left(\int_0^{\alpha_i} k\,\ytilde_i(\theta)\dtheta
                  \,+\, \int_{\alpha_i}^{2\pi} \ytilde_i(\theta)\dtheta\right)
              \delta^{\gamma_{i1}-1}
              \,+\, o(1)\,.
\edmal
The two integrals over $\ytilde_i$ in the bottom line have already been
evaluated in the course of the proof of Theorem~\ref{Thm:HeRu98-1},
compare~\req{integral-identity}. Their value is such that all terms of 
order $\delta^{\gamma_{i1}-1}$ cancel, proving that
\bdm
   \scalp{w,g}_{L^2(\partial\Omega)}
   \,=\, (1-k) \int_{\partial\D} (h\cdot\nu)\,\nabla u^- \cdot(M\nabla v_g^-)\ds
         \,+\, o(1)\,, \qquad \delta\to 0\,.
\edm
Since the left-hand side of this equation is independent of $\delta$, the
desired identity~\req{Thm:BFV17} must hold true.

As shown in Theorem~\ref{Thm:domderiv} the shape derivative $u_h'$ of $u$
is harmonic in $\D$ and in $\Omega\setminus\overline\D$ and has homogeneous
Neumann boundary values. Accordingly, $u_h'$ and $w$ 
share the same Cauchy data on $\partial\Omega$. Since they
are both harmonic in $\Omega\setminus\overline\D$ they coincide up to the
boundary of $\D$ by virtue of Holmgren's theorem. In particular, their
restrictions to $\Omega\setminus\overline\D$ have
the same trace on $\partial\D$. Since the material derivative 
$\upunkt_{H}$ associated with the extension $H$ of $h$ described
in the beginning of Section~\ref{Sec:estimates}
belongs to $H^1(\Omega)$ and
the traces of $u^\pm$ (and their tangential derivatives) coincide on 
$\partial\D$ it follows from \req{Thm:domderiv} and \req{Bnu} that
\begin{align*}
   \bigl[u_h'\bigr]_{\partial\D}
   &\,=\, -H|_{\partial\D}\cdot\bigl[\nabla u\bigr]_{\partial\D}
    \,=\, -h\cdot\bigl[\nabla u\bigr]_{\partial\D}
    \,=\, -(h\cdot\nu)\Bigl(\frac{\partial}{\partial\nu}u^+
                           \,-\,\frac{\partial}{\partial\nu}u^-\Bigr)\\
   &\,=\, (1-k)(h\cdot\nu)\frac{\partial}{\partial\nu}u^-\,.
\end{align*}
According to \req{HeRu98-2} this shows that $w$ and $u_h'$ experience the
same jump across $\partial\D$, and therefore the traces on $\partial\D$ of
$w|_\D$ and $u_h'|_\D$ also coincide. Both being harmonic functions
we thus conclude that $w=u_h'$ in $\D$, and hence, in all of $\Omega$.
\end{proof}

%%%%%%%%%%%%%%%%%%%%%%%%%%%%%%%%%%%%%%%%%%%%%%%%%%%%%%%%%%%%%%%%%%%%%%%
\section{The case of an insulating or a perfectly conducting polygonal 
inclusion}
\label{Sec:degeneratecases}
%%%%%%%%%%%%%%%%%%%%%%%%%%%%%%%%%%%%%%%%%%%%%%%%%%%%%%%%%%%%%%%%%%%%%%% 
So far we have assumed that the conductivity of the inclusion is a known
positive value $k>0$. Now we turn to the limiting cases which are formally
described by $k=0$ (the insulating case) and $k=+\infty$ (the perfectly
conducting case).

%%%%%%%%%%%%%%%%%%%%%%%%%%%%%%%%%%%%%%%%%%%%%%%%%%%%%%%%%%%%%%%%%%%%%%%
\subsection{The insulating case}
\label{Subsec:insulating}
%In the insulating case we can actually set $k=0$ in \req{forwardvarproblem},
%and consider this variational form for 
%$u,v\in H^1(\Omega\setminus\overline\D)$. 
The appropriate formulation of the 
conductivity equation~\req{conductivity-equation} when $\D$ is an insulating
inclusion, is in terms of the boundary value problem
\be{insulating-equation}
\begin{array}{c}
   \Delta u \,=\, 0 \quad \text{in $\Omega\setminus\overline\D$}\,, \qquad
   {\displaystyle \int_{\partial\Omega} u\ds \,=\, 0}\,,\\[2ex]
   \dfrac{\partial}{\partial\nu}u \,=\, 0 \quad \text{on $\partial\D$}\,,
   \qquad
   \dfrac{\partial}{\partial\nu}u \,=\, f \quad \text{on $\partial\Omega$}\,.
\end{array}
\ee
When $\D$ is smooth the existence of an associated shape derivative of 
$u=u(\D)$ is implicit in the result of \cite[Section~3.2]{SoZo92}:
this shape derivative in direction $h\in\S$ is given by the 
solution $u_h'\in H^1(\Omega\setminus\overline\D)$ of the 
Neumann boundary value problem
\be{uprime-insulating}
\begin{array}{c}
   \Delta u_h' \,=\, 0 \quad \text{in $\Omega\setminus\overline\D$}\,, \qquad
   {\displaystyle \int_{\partial\Omega} u_h'\ds \,=\, 0}\,, \\[2ex]
   \dfrac{\partial}{\partial\nu} u_h' \,=\, 
      \dfrac{\partial}{\partial\tau}
         \Bigl((h\cdot\nu)\dfrac{\partial}{\partial\tau} u\Bigr)
      \quad \text{on $\partial\D$}\,, \qquad
   \dfrac{\partial}{\partial\nu} u_h' \,=\, 0 \quad 
   \text{on $\partial\Omega$}\,.
\end{array}
\ee
To the best of our knowledge the existence of a corresponding shape derivative
for a polygonal inclusion $\D$ has not yet been investigated. 
However, the analysis in \cite{BMPS18} can be adapted in a 
straightforward manner to derive the respective material derivative 
$\upunkt_{H}\in H^1(\Omega\setminus\overline\D)$ in the direction of
any $H\in\W$
for the insulating case and the corresponding variational definition
of $\partial\Lambda_f(\D)h=\upunkt_{H}|_{\partial\Omega}$,
when $h=H|_{\partial\D}$, namely
%a natural modification 
%of the variational definition~\req{BFV17} for the case $k=0$, namely
\bdm
   \scalp{\partial\Lambda_f(\D)h,g}_{L^2(\partial\Omega)}
   \,=\, \int_{\partial\D} (h\cdot\nu)\,
            \frac{\partial}{\partial\tau} u \,
            \frac{\partial}{\partial\tau}v_g\ds
\edm
for every $g\in L^2_\diamond(\partial\Omega)$, where $v_g$ is the corresponding
solution of \req{insulating-equation} with $f$ replaced by $g$.
This is the natural analog of \req{BFV17} for $k=0$. 
%Note that in the
%insulating case the material derivative satisfies
%\bdm
%   \frac{\partial}{\partial\nu} \upunkt
%   \,=\, \lim_{t\to 0} \frac{\partial}{\partial\nu}\frac{\utilde_t-u}{t}
%   \,=\, 0 \qquad \text{on $\partial\D$}\,.
%\edm

For the proof of the corresponding version of Theorem~\ref{Thm:HeRu98-1} 
analogous properties of the 
solution of the forward problem~\req{insulating-equation}, as recollected in 
Section~\ref{Sec:forward} for the case $k>0$, are needed. 
The corresponding results can be found, e.g., in \cite{Gris92}:
In the case $k=0$ the eigenvalues $\gamma_{ij}^2$ and eigenfunctions $y_{ij}$ 
which enter into the expansion of $u$ near any of the vertices,
are known explicitly, cf.~\cite[p.~50]{Gris92}, namely
\be{gamma-Grisvard}
   \gamma_{ij} \,=\, j\pi/(2\pi-\alpha_i)\,, \qquad j\in\N_0\,,
\ee
and, for $j\in\N$,
\bdm
   y_{ij}(\theta) \,=\, \Bigl(\frac{2}{2\pi-\alpha_i}\Bigr)^{1/2}
                       \cos\gamma_{ij}(\theta-\alpha_i)\,, 
   \qquad \alpha_i<\theta<2\pi\,.
\edm
Take note that the nonnegative roots $\gamma_{ij}$
of these eigenvalues still satisfy \req{gamma} when $0<\alpha_i<\pi$, 
whereas they are all greater than one, if $\pi<\alpha_i<2\pi$. 
Using these properties of the solution of the forward problem
the proof of Theorem~\ref{Thm:HeRu98-1} carries over to the insulating case 
with straightforward modifications; 
for example, the singular subspaces $S_i$ are now two-dimensional, each, 
and only contain the restrictions to $\Omega\setminus\overline\D$ of the 
singular functions $w_i$ of \req{uiprime}.
Since at least one of the interior angles of $\D$ is smaller than $\pi$, 
the corresponding singular functions
only belong to $H^\gamma(\Omega\setminus\overline\D)$ for some 
$\gamma\in(1/2,1)$ according to \cite[Theorem~1.2.18]{Gris92}, 
and hence the overall smoothness of the solution $u_h'$ of
\req{uprime-insulating} is of similar type as in the case
of a polygonal inclusion with positive conductivity.

If the inhomogeneous Neumann boundary data on $\partial\D$
in \req{uprime-insulating} are defined by some vector field $h\in\S$,
a little extra care is necessary to show that the corresponding solution 
of \req{uprime-insulating} is the shape derivative of $u$ in direction $h$,
because the function $(u_{th}-u)/t$ is not defined in all of
$\Omega$; in fact, the two functions $u_{th}$ and $u$ live on different domains.
To overcome this problem one can consider an arbitrary domain 
$\Omega'\subset\Omega\setminus\overline\D$
which satisfies $\partial\Omega'\cap\partial\D=\emptyset$, in which case
\bdm
   \frac{u_{th}-u}{t}\,:\,\Omega'\to\R
\edm
is well-defined for $t$ sufficiently close to zero. 
The analog of Theorems~\ref{Thm:domderiv}, \ref{Thm:HeRu98-1}, and 
\ref{Thm:domderivPDE} then reads as follows.

\begin{theorem}
\label{Thm:domderiv-insulating}
Let $\D$ be a polygon with simply connected closure $\overline\D\subset\Omega$
and $h\in\S$. 
Then the boundary value problem~\req{uprime-insulating} has a unique solution 
$u'\in H^1(\Omega\setminus\overline\D)\oplus S_1\oplus\dots\oplus S_n
\subset H^\gamma(\Omega\setminus\overline\D)$ for some $\gamma\in(1/2,1)$.
This solution is the shape derivative of $u=u(\D)$ in the direction of $h$. 
This means that if $H\in\W$ is the extension of $h$ described in the 
beginning of Section~\ref{Sec:estimates}, then
\bdm
   \frac{u_{th}-u}{t} \,\to\, u_h' 
   \,=\, \upunkt_{H} \,-\, H\cdot\nabla u\,, 
   \qquad t\to 0\,,
\edm
in $L^2(\Omega')$ for every domain $\Omega'\subset\Omega$ with
$\partial\Omega'\cap\partial\D=\emptyset$.
%This shape derivative $\red{u_h'}$ also belongs to 
%$H^1\bigl(\Omega\setminus\overline{\D\cup_i\B_{r_i}(x_i)}\bigr)$, and 
The trace of $u_h'$ on $\partial\Omega$ is the shape derivative of 
$\Lambda_f(\D)$ in direction $h$.
\end{theorem}

The other results of Section~\ref{Sec:estimates} apply 
verbatim to the insulating case;
%Like the analysis in \cite{BMPS18} 
the corresponding modifications of the proofs are straightforward.

%%%%%%%%%%%%%%%%%%%%%%%%%%%%%%%%%%%%%%%%%%%%%%%%%%%%%%%%%%%%%%%%%%%%%%%
\subsection{The perfectly conducting case}
\label{Subsec:conducting}
When $k$ becomes infinitely large, on the other hand, the
transmission condition~\req{Bnu} forces $u^-$ to approximate a
constant but unknown value; this is in agreement with our intuition that
in a perfect conductor potential differences equilibrate immediately. 
Up to an additive constant the corresponding potential 
$u\in H^1(\Omega\setminus\overline\D)$ is therefore given by the
solution of the boundary value problem
\be{superconducting-equation}
   \Delta u\,=\, 0 \quad \text{in $\Omega\setminus\overline\D$}\,, \qquad
   u\,=\,0 \quad \text{on $\partial\D$}\,, \qquad 
   \frac{\partial}{\partial\nu}u \,=\, f \quad \text{on $\partial\Omega$}\,.
\ee
Note that $u$ can be extended by zero to a function in $H^1(\Omega)$.
Ito, Kunisch, and Li~\cite{IKL01} pointed out that 
%for a perfectly conducting inclusion $\D$ with smooth boundary the trace of 
the solution of the boundary value problem
\be{IKL01}
\begin{array}{c}
   \Delta u_h' \,=\, 0 \quad \text{in $\Omega\setminus\overline\D$}\,, \\[2ex]
   u_h' \,=\, - (h\cdot\nu)\dfrac{\partial}{\partial\nu} u
      \quad \text{on $\partial\D$}\,, \qquad
   \dfrac{\partial}{\partial\nu} u_h' \,=\, 0 \quad 
   \text{on $\partial\Omega$}\,.
\end{array}
\ee
is the shape derivative of $u=u(\D)$ of \req{superconducting-equation}
in the direction of $h\in\S$, provided that $\D$ has a smooth boundary.

When $\D$ is a polygon the solution of the 
forward problem~\req{superconducting-equation} admits a similar expansion
near the vertices of $\D$ as in the insulating case; 
again, see \cite{Gris92} for details.
The eigenvalues $\gamma_{ij}^2$ which are relevant for this expansion
are the same as in the insulating case, cf.~\req{gamma-Grisvard}, 
this time the eigenfunctions being the corresponding sine functions
\bdm
   y_{ij}(\theta) \,=\, \Bigl(\frac{2}{2\pi-\alpha_i}\Bigr)^{1/2}
                       \sin \gamma_{ij}(\theta-\alpha_i)\,,
   \qquad \alpha_i<\theta<2\pi\,,
\edm
where the index $j$ runs through the natural numbers only; all eigenvalues
are strictly positive.

The analysis in \cite{BMPS18} also extends to the perfectly conducting
case, showing that the shape derivative $\partial\Lambda_f(\D)$ of
a polygonal perfect conductor $\D$ exists, and that it satisfies 
\bdm
   \scalp{\partial\Lambda_f(\D)h,g}_{L^2(\partial\Omega)}
   \,=\, -\int_{\partial\D} (h\cdot\nu)\,
             \frac{\partial}{\partial\nu} u\,
             \frac{\partial}{\partial\nu}v_g\ds
\edm
for every $g\in L^2_\diamond(\partial\Omega)$, where $v_g$ solves the boundary
value problem~\req{superconducting-equation} with $f$ replaced by $g$.
Note that the projection of $\partial\Lambda_f(\D)h$ onto 
$L^2_\diamond(\partial\Omega)$ is independent of the chosen grounding 
in the definition~\req{superconducting-equation} of $u$ and $v_g$, 
respectively.

Further, the analysis from Sections~\ref{Sec:estimates} to
\ref{Sec:domderiv} can be modified in a straightforward way to achieve 
the following result.

\begin{theorem}
\label{Thm:uprime-perfectconducting}
Let $\D$ be a polygon with simply connected closure 
$\overline\D\subset\Omega$ and $h\in\S$.
Then the shape derivative of the solution $u=u(\D)$ of 
\req{superconducting-equation} in direction $h$ is given by the 
unique solution 
$u_h'\in H^1(\Omega\setminus\overline\D)\oplus S_1\oplus\dots\oplus S_n$ 
of the boundary value problem~\req{IKL01}.
This solution belongs to $H^{\gamma}(\Omega\setminus\overline\D)$ 
for some $\gamma\in(1/2,1)$ and 
%to $H^1\bigl(\Omega\setminus\overline{\D\cup_i\B_{r_i}(x_i)}\bigr)$.
its trace on $\partial\Omega$ is the shape derivative of $\Lambda_f(\D)$
in direction $h|_{\partial\D}$.
\end{theorem}

Like in the insulating case the sets $S_i$ are two-dimensional, each, and
the other results of Section~\ref{Sec:estimates}
are valid in the perfectly conducting case, too.

%%%%%%%%%%%%%%%%%%%%%%%%%%%%%%%%%%%%%%%%%%%%%%%%%%%%%%%%%%%%%%%%%%%%%%%
\section*{Acknowledgement}
%%%%%%%%%%%%%%%%%%%%%%%%%%%%%%%%%%%%%%%%%%%%%%%%%%%%%%%%%%%%%%%%%%%%%%% 
The author is grateful to Elena Beretta for pointing out 
references~\cite{BMPS18,Laur20} as a possible starting point to determine
the shape derivative $\partial\Lambda_f$ for insulating or perfectly
conducting polygonal inclusions. Thanks also go to the referees: their
comments helped to improve the presentation.

%%%%%%%%%%%%%%%%%%%%%%%%%%%%%%%%%%%%%%%%%%%%%%%%%%%%%%%%%%%%%%%%%%%%%%% 

\end{document}